\documentclass[letterpaper]{article}

\usepackage{amsfonts,amsmath,amssymb,amsthm}
\usepackage{graphicx}
\usepackage{epsfig}
\usepackage{eso-pic}
\usepackage{color}
\usepackage{type1cm}
\usepackage[utf8x]{inputenc} % To use Unicode (e.g. Turkish) characters
\usepackage{dsfont}
\usepackage{mathrsfs}
\usepackage{yfonts}

\usepackage[bottom]{footmisc}

\usepackage{longtable}
\usepackage{multirow}
\usepackage{subfigure}
\usepackage{algorithm}
\usepackage{algorithmic}
\usepackage{mathtools}
\usepackage[letterpaper]{geometry}

\newcommand*{\itone}{I_{IT_1}} 
\newcommand*{\ittwo}{I_{IT_2}}

\newcommand*{\stone}{I_{ST_1}}  
\newcommand*{\sttwo}{I_{ST_2}}

\newcommand*{\il}{I_{IL}}  
\newcommand*{\ds}{I_{DS}}

\newcommand*{\sbone}{I_{SB_1}}  
\newcommand*{\sbtwo}{I_{SB_2}}

\newcommand*{\N}{\mathbb{N}}

\newcommand*{\Pol}{\mathbb{P}}
\newcommand*{\Tri}{\mathbb{T}}

\newcommand*{\charfunc}{\mathds{1}}

\newtheorem{theorem}{Theorem}

\newtheorem{corollary}[theorem]{Corollary}

\newtheorem{proposition}[theorem]{Proposition}
\newtheorem{remark}[theorem]{Remark}

\begin{document}

\title{A Galerkin BEM for high-frequency scattering problems based on
frequency dependent changes of variables}
\author{
Fatih Ecevit\thanks{Corresponding author. Email: fatih.ecevit@boun.edu.tr}
\\[2pt]
{\small Bo\u{g}azi\c{c}i University, Department of Mathematics, Bebek TR 34342, Istanbul, Turkey}
\\[6pt]
and
\\[6pt]
Hasan H\"{u}seyin Eruslu\thanks{Email: heruslu@udel.edu}
\\[2pt]
{\small University of Delaware, Department of Mathematical Sciences, Newark, DE 19716, USA}
}

\date{}

\maketitle

\begin{abstract}
In this paper we develop a class of efficient Galerkin boundary element methods for the solution
of two-dimensional exterior single-scattering problems. Our
approach is based upon construction of Galerkin approximation spaces confined to the asymptotic
behaviour of the solution through a certain direct sum of appropriate function spaces weighted by
the oscillations in the incident field of radiation. Specifically, the function spaces in the
illuminated/shadow regions and the shadow boundaries are simply algebraic polynomials whereas
those in the transition regions are generated utilizing novel, yet simple, \emph{frequency dependent
changes of variables perfectly matched with the boundary layers of the amplitude} in these regions.
While, on the one hand, we rigorously verify for smooth convex obstacles that these methods require
only an $\mathcal{O}\left( k^{\epsilon} \right)$ increase in the number of degrees of freedom to
maintain any given accuracy independent of frequency, and on the other hand, remaining in the
realm of smooth obstacles they are applicable in more general single-scattering configurations.
The most distinctive property of our algorithms is their \emph{remarkable success} in approximating
the solution in the shadow region when compared with the algorithms available in the literature.
\end{abstract}

\section{Introduction}
High-frequency scattering problems have found and continue to find immense interest
in the present day computational science. Indeed, over the course of last two decades,
very efficient and effective algorithms have been devised for the numerical solution of
scattering problems based on variational
\cite{HesthavenWarburton04,DaviesEtAl09,Boffi10}
and integral equation
\cite{BrunoKunyansky01,BanjaiHackbusch05,TongChew10,BrunoEtAl13}
formulations. In exterior scattering simulations, methods that rest on variational formulations
naturally demand the design and implementation of efficient non-reflecting boundary
conditions \cite{EnquistMajda77, Givoli04, GroteSim11} to effectively represent the radiation
condition at infinity. On the other hand, solvers based on integral equation formulations
(cf. the survey \cite{Chandler-WildeEtAl12}) readily encode the radiation condition
into the equation by choosing an outgoing fundamental solution. Moreover, for surface
scattering simulations considered in this paper, they enable \emph{phase extraction},
that takes a particularly simple form in single-scattering configurations, and this turns
the problem into the estimation of an amplitude whose oscillations are essentially
independent of frequency.

In this paper, we develop a class of efficient Galerkin boundary element
methods for the solution of two-dimensional exterior single-scattering problems
Our approach is based upon construction of Galerkin
approximation spaces confined to the known asymptotic behaviour of the aforementioned
amplitude. These spaces are defined as direct sums of appropriate function spaces
weighted by the oscillations in the incident field of radiation. Specifically, the function
spaces in the illuminated/shadow regions and the shadow boundaries are simply algebraic
polynomials whereas those in the transition regions (filling up the remaining parts of the
boundary of the scatterer) are generated utilizing novel, yet simple, \emph{frequency dependent
changes of variables perfectly matched with the boundary layers of the amplitude} in these
regions. While, on the one hand, we rigorously verify for smooth convex obstacles that
these methods require only an $\mathcal{O}\left( k^{\epsilon} \right)$ increase in the number
of degrees of freedom to maintain any given accuracy independent of frequency, and on
the other hand, remaining in the realm of smooth obstacles they are applicable in more
general single-scattering configurations. The most distinctive property of our algorithms
is their \emph{remarkable success} in approximating the solution in the shadow region
when compared with the algorithms available in the literature.

Indeed, hybrid integral equation methodologies reinforcing the asymptotic characteristics
of the unknown into the solution strategy have now become the usual practice in the field. 
The first attempt in this direction is due to N\'{e}d\'{e}lec  et al. \cite{AbboudEtAl94,AbboudEtAl95}
where, considering the impedance boundary condition, an $h$-version boundary element
method was utilized in conjunction with the method of stationary phase
for the evaluation of highly oscillatory integrals (see \cite{Darrigrand02} for a
fast multipole variant, and \cite{GaneshHawkins11} for a fully discrete
three-dimensional version). More relevant to our work is the Nystr\"{o}m method proposed
by Bruno et al. \cite{BrunoEtAl04} for the solution of sound-soft scattering problems in the
exterior of smooth convex obstacles. The method therein displays the capability of delivering
solutions within any prescribed accuracy in frequency independent computational times.
While, in this approach, the boundary layers of the slowly varying amplitude around the shadow boundaries
are resolved through a cubic root change of variables, the associated highly oscillatory
integrals are evaluated to high order utilizing novel extensions of the method of stationary
phase (for a three-dimensional variant of this approach we refer to
\cite{BrunoGeuzaine07}). The algorithm in \cite{BrunoEtAl04} has had a
great impact in computational scattering community and, following the basic principles
therein, a number of alternative single-scattering solvers have been developed.
Giladi \cite{Giladi07} have used a collocation method that integrates Keller's geometrical
theory of diffraction to account for creeping rays in the shadow region. Huybrechs et al.'s
collocation method \cite{HuybrechsVandewalle07} have utilized the numerical steepest
descent method in evaluating highly oscillatory integrals and additional collocation points
around shadow boundaries to obtain sparse discretizations. The first rigorous numerical
analysis relating to a p-version boundary element implementation of these approaches,
due to Dom\'{i}nguez et al. \cite{DominguezEtAl07}, has displayed that an increase of
$\mathcal{O}(k^{1/9})$ in the number of degrees of freedom is sufficient to preserve a
certain accuracy as $k \to \infty$.

The aforementioned methods remain asymptotic due to approximation of the solution
by zero in the deep shadow region. In order to cure this deficiency, we have recently
developed \emph{frequency-adapted Galerkin boundary element methods} \cite{EcevitOzen16}.
Our approach therein was based on utilization of appropriate integral equation
formulations of the scattering problem and design of Galerkin approximation spaces
as the direct sum of algebraic or trigonometric polynomials weighted by the oscillations
in the incident field of radiation. The number of direct summands, namely $4m$ (one for
each of the illuminated and deep shadow regions, and the two shadow boundaries; and
$m-1$ for each one of the four transition regions), had to increase as $\mathcal{O}(\log k)$
in order to obtain optimal error bounds. Moreover, from a theoretical perspective, we
have rigorously shown that these methods can be tuned to demand an increase of
$\mathcal{O}(k^{\epsilon})$ in the number of degrees of freedom to maintain a
prescribed accuracy independent of frequency. In connection with smooth convex
scatterers, this is the best theoretical result available in the literature.

The new Galerkin boundary element methods we develop in
this paper, in contrast, utilize novel \emph{frequency dependent changes of
variables perfectly matched with the asymptotic behaviour of the solution in the
transition regions} and thereby eliminate the requirement of increasing the number
of direct summands defining the Galerkin approximation spaces with increasing
wavenumber $k$. While this clearly displays the ease of implementation of these
new schemes when compared with our approach in \cite{EcevitOzen16}, as we
rigorously verify, an increase of $\mathcal{O}(k^{\epsilon})$ is still sufficient to fix the
approximation error with increasing $k$ but with savings of $\mathcal{O}(\sqrt{\log k})$
in the necessary number of degrees of freedom. Perhaps more importantly, these
new schemes yield \emph{significantly superior accuracy in the shadow region} as depicted
through the numerical tests. This is also true when the results are compared with
the recent approach taken by Asheim and Huybrechs \cite{AsheimHuybrechs14}
wherein more advanced phase extraction techniques (unfortunately not supported
by rigorous numerical analysis) based on Melrose-Taylor asymptotics
\cite{MelroseTaylor85} are used.

Parallel with the schemes relating to smooth convex obstacles, Galerkin boundary
element methods based on phase extraction have also been developed for half-planes
and convex polygons where the number of degrees of freedom is either fixed or
must increase in proportion to $\log k$ to fix the error with increasing wavenumber
$k$. For an extended review of these procedures, we refer to the survey article
\cite{Chandler-WildeEtAl12}.

As for multiple scattering problems, we refer to Bruno et al. \cite{BrunoEtAl05}
for an extension of the algorithm in \cite{BrunoEtAl04} to a finite collection of
convex obstacles (see also \cite{EcevitReitich09} and \cite{AnandEtAl10} for a rigorous
analysis of this approach in two- and three-dimensional settings respectively,
and Boubendir et al. \cite{BoubendirEtAl16} for the acceleration of this procedure through use
of \emph{dynamical} Krylov subspaces and Kirchhoff approximations), and to
Chandler-Wilde et al. \cite{Chandler-WildeEtAl15} for a class of nonconvex
polygons. 

The paper is organized as follows. In \S\ref{sec:2}, we describe the exterior
sound-soft scattering problem along with the relevant integral equations and
associated Galerkin formulations. In \S\ref{sec:3}, we introduce the new Galerkin
schemes for high-frequency single-scattering problems, and state the associated
convergence theorem for smooth convex obstacles which constitutes the main
result of the paper. To allow a direct comparison, in the same section, we also
present a more general version of our algorithm
in \cite{EcevitOzen16} along with the corresponding approximation properties.
The proof of the main result of the paper is given in \S\ref{sec:4}. Finally, the
numerical tests appearing in \S\ref{sec:5} provide a comparison of our methods
developed herein and \cite{EcevitOzen16}. Specifically they display that these
new schemes 1) attain the same global accuracy with a reduced number of
degrees of freedom, 2) provide significantly more accurate solutions in the shadow
regions, and 3) are applicable not only for smooth convex obstacles but also in
more general single-scattering configurations.
 
\section{The scattering problem and Galerkin formulation}
\label{sec:2}

The two-dimensional scattering problem we consider in this manuscript is related
with the determination of the \emph{scattered field} $u$ that satisfies the Helmholtz
equation
\[
	\left( \Delta + k^{2} \right) u = 0
\]
in the exterior of a smooth compact obstacle $K$, the Sommerfeld radiation
condition at infinity that amounts to requiring
\[
	\lim_{\left| x \right| \to \infty} \left| x \right|^{1/2} \left[ \partial_{\left| x \right|} - i k \right] u = 0
\]
uniformly for all directions $x/\left| x \right|$, and the sound-soft boundary condition
\[
	u = - u^{\rm inc}	
\]
for a plane-wave incidence $u^{\rm inc} \left( x \right) = e^{ik \alpha \cdot x}$
with direction $\alpha$ ($\left| \alpha \right| =1$) impinging on $K$.

As is well known, the scattered field $u$ can be reconstructed by means of either
the \emph{direct} or the \emph{indirect approach} \cite{ColtonKress92}. As in the
previous attempts aimed at frequency-independent simulations
\cite{BrunoEtAl04,Giladi07,HuybrechsVandewalle07,DominguezEtAl07,EcevitOzen16},
however, here we favour the former wherein the associated (unknown) surface density
is the \emph{normal derivative of the total field} (known as the \emph{surface current}
in electromagnetism) $\eta = \partial_{\nu} \left( u+ u^{\rm inc} \right)$ on $\partial K$.
Once $\eta$ is available, the scattered field can be recovered through the
\emph{single-layer potential}
\[
	u(x) = - \int_{\partial K} \Phi(x,y) \, \eta(y) \, ds(y)
\]
where
\[
	\Phi(x,y) = \dfrac{i}{4} \ H_{0}^{(1)}(k|x-y|)
\]
is the fundamental solution of the Helmholtz equation, and $H_{0}^{(1)}$ is the
Hankel function of the first kind and order zero. 

While the new unknown, namely $\eta$, can be expressed as the unique solution
of a variety of integral equations \cite{ColtonKress92,SpenceEtAl11} taking the
form of an operator equation
\begin{equation}
\label{eq:CFIE}
	\mathcal{R}_k \, \eta = f_k
\end{equation}
in $L^{2} \left( \partial K \right)$, 
the solution of (\ref{eq:CFIE}) corresponds exactly to that of
\begin{equation}
\label{eq:weakformulation}
	B_k(\mu,\eta)
	= F_k(\mu),
	\quad
	\text{for all }
	\mu \in L^{2}(\partial K),
\end{equation}
where the sesquilinear form $B_k$ and bounded linear functional $F_k$ are
defined by
\begin{equation}
\label{eq:sesqui}
	B_k(\mu,\eta) = \langle \mu, \mathcal{R}_k \, \eta \rangle_{L^{2}(\partial K)}
	\quad
	\text{and}
	\quad
	F_k(\mu) = \langle \mu, f_k \rangle_{L^{2}(\partial K)}.
\end{equation}
Equation \eqref{eq:weakformulation}, in turn, is amenable to a treatment by the
Galerkin method wherein one determines the \emph{Galerkin solution} $\hat{\eta}$
approximating the exact solution $\eta$ in a given finite dimensional \emph{Galerkin
subspace} $\hat{X}_k$ requiring that
\begin{equation}
\label{eq:Galerkinsolution}
	B_k(\hat{\mu},\hat{\eta})
	= F_k(\hat{\mu}),
	\quad
	\text{for all }
	\hat{\mu} \in \hat{X}_k.
\end{equation}  
Further, provided the sesquilinear form $B_k$ is continuous with a continuity
constant $C_k$ and strictly coercive with a coercivity constant $c_k$ so that
\[
	\left| B_k(\mu,\eta) \right| \le C_k \, \Vert \mu \Vert \Vert \eta \Vert
	\qquad
	\text{and}
	\qquad
	\mathcal{R}e \ B_k(\mu,\mu) \ge c_k \, \Vert \mu \Vert^{2}
\]
for all $\mu, \eta \in L^2(\partial K)$, equation \eqref{eq:Galerkinsolution} is
uniquely solvable and C$\acute{e}$a's lemma entails
\begin{equation}
\label{eq:Cea}
	\Vert \eta -\hat{\eta} \Vert
	\leq \dfrac{C_k}{c_k}
	\inf_{\hat{\mu} \in \hat{X}_k} \Vert \eta- \hat{\mu} \Vert.
\end{equation}
Among the aforementioned integral equations, in this connection, \emph{combined
field} (CFIE) and \emph{star-combined} (SCIE) integral equations step forward as
the continuity and coercivity properties of the associated sesquilinear forms are
well-understood. More precisely, the sesquilinear form associated with CFIE is
known to be continuous (for $k>0$) and coercive (for $k \gg 1$) for convex domains
with piecewise analytic $C^3$ boundaries with $C_k = \mathcal{O} \left( k^{1/2} \right)$
as $k \to \infty$ and $c_k \ge 1/2$ for all sufficiently large $k$ \cite{DominguezEtAl07,SpenceEtAl15}
(see also \cite{BaskinEtAl16} for an extension to non-trapping domains).
On the other hand, the sesquilinear form corresponding to SCIE is both continuous
and coercive (for $k>0$) for star-shaped Lipschitz domains with
$C_k = \mathcal{O} \left( k^{1/2} \right)$ as $k \to \infty$ and $c_k$ independent of
$k$ \cite{SpenceEtAl11}.

\section{Galerkin approximation spaces based on frequency dependent changes of variables}
\label{sec:3}

The developments in this section are independent of the integral equation used
as they relate, specifically, to the construction of Galerkin approximation spaces
$\hat{X}_k$ whose dimension should increase only as $\log k$ with increasing
wave number $k$ to ensure that the relative error
\[
	\inf_{\hat{\mu} \in \hat{X}_k}  \dfrac{\Vert \eta- \hat{\mu} \Vert}{\Vert \eta \Vert}
\]
in connection with the infimum on the right-hand side of \eqref{eq:Cea} is independent
of $k$.

Considering a smooth convex obstacle $K$ illuminated by a plane-wave
$u^{\rm inc} (x) = e^{ik \alpha \cdot x}$, our approach is based on 
\emph{phase extraction}
\[
	\eta \left( x \right) = e^{ik \alpha \cdot x} \, \eta^{\rm slow} \left( x \right),
	\qquad
	x \in \partial K,
\]
and design of approximation spaces adopted to the asymptotic behavior of the amplitude
$\eta^{\rm slow}$ that was initially characterised by Melrose and Taylor \cite{MelroseTaylor85}
around the shadow boundaries which we have later generalized to the entire boundary
\cite{EcevitReitich09}. 

\begin{theorem}\cite[Corollary 2.1]{EcevitReitich09}
\label{thm:ecevit}
Let $K \subset \mathbb{R}^{2}$ be a compact, strictly convex set with smooth boundary
$\partial K$. Then $\eta^{\rm slow} = \eta^{\rm slow}(x,k)$ belongs to the H\"{o}rmander
class $S^{1}_{2/3,1/3} \left( \partial K \times \left( 0, \infty \right) \right)$ and admits an
asymptotic expansion
\[
	\eta^{\rm slow}(x,k)
	\sim
	\sum_{p,q \ge 0} a_{p,q}\left( x,k \right)
\]
with
\[
	a_{p,q}(x,k) = k^{2/3-2p/3-q} \, b_{p,q}(x)  \, \Psi^{(p)}(k^{1/3}Z(x))
\]
where $b_{p,q}$ and $\Psi$ are complex-valued $C^{\infty}$ functions and $Z$
is a real-valued $C^{\infty}$ function that is positive on the illuminated region
$\partial K^{IL} = \{ x \in \partial K : \alpha \cdot \nu(x) < 0 \}$, negative on the
shadow region $\partial K^{SR} = \{ x \in \partial K : \alpha \cdot \nu(x) > 0 \}$,
and vanishes precisely to the first order on the shadow boundaries
$\partial K^{SB} = \{ x \in \partial K : \alpha \cdot \nu(x) = 0 \}$. Moreover, the
function $\Psi$ admits the asymptotic expansion
\[
	\Psi(\tau) \sim \sum_{j=0}^{\infty} c_{j} \tau^{1-3j}
	\qquad
	\text{as } \tau \to \infty,
\]
and $\Psi$ is rapidly decreasing in the sense of Schwartz as
$\tau \to -\infty$.
\end{theorem}

\begin{figure}[htbp]
\begin{center}
\hbox{
	\subfigure[Regions on the unit circle]
		{\includegraphics*[height=3.65cm, trim={2.5cm 0.7cm 5.5cm 1.4cm},scale=0.3,clip]{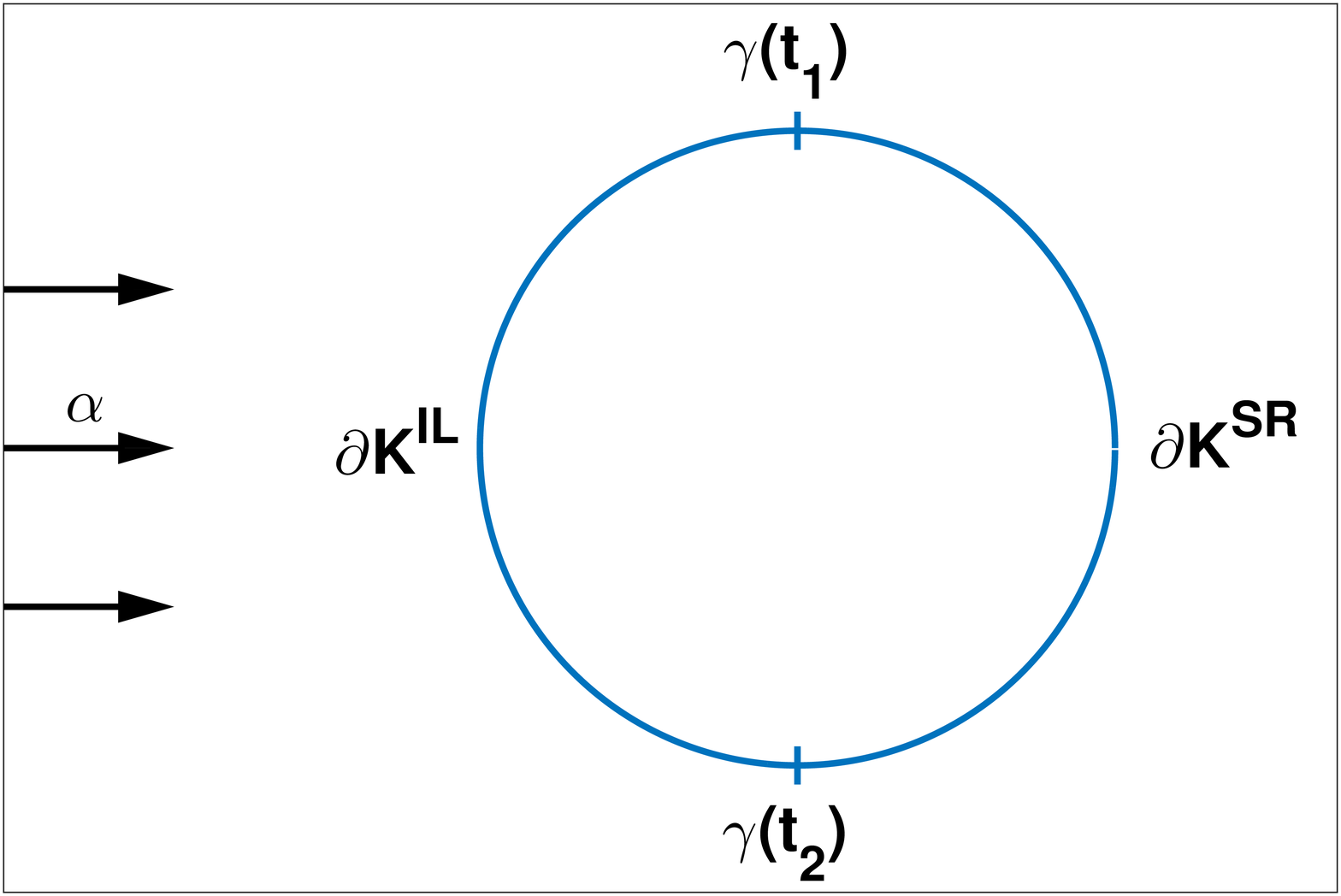}}
	\subfigure[Boundary layers generated]
		{\includegraphics*[height=3.65cm, trim={2.5cm 0.7cm 5.5cm 1.4cm},scale=0.3,clip]{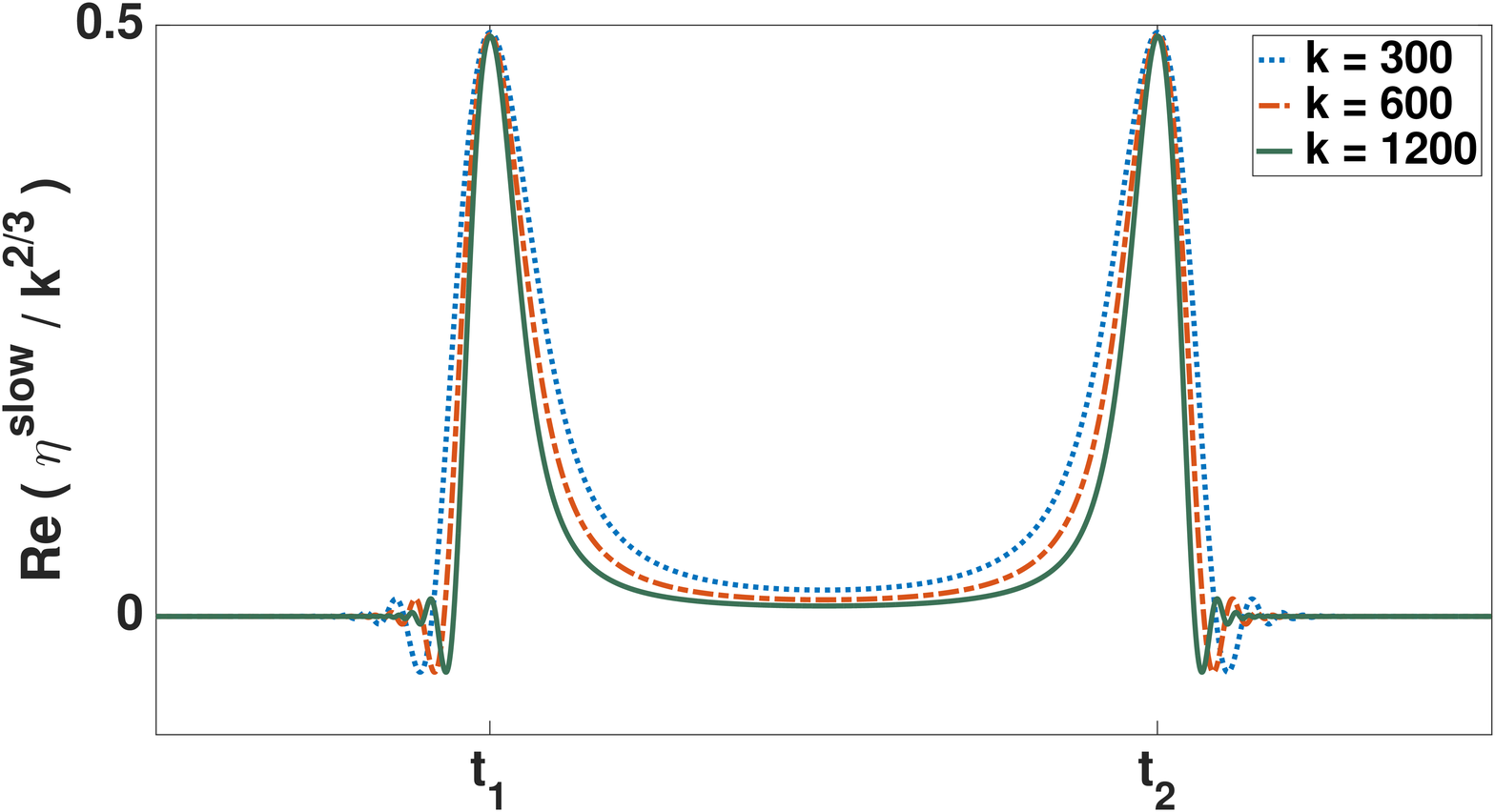}}
}
\end{center}
\caption{The unit circle illuminated from the left and the associated boundary layers of the solutions.}
\label{fig:RegionsLayers}
\end{figure}

Theorem~\ref{thm:ecevit} clearly displays the challenges associated with the
construction of approximation spaces adapted to the asymptotic behavior of
$\eta^{\rm slow}$ as it shows that while $\eta^{\rm slow}$ admits a classical
asymptotic expansion in the illuminated region and rapidly decays in the
shadow region, it possesses boundary layers in the $\mathcal{O}(k^{-1/3})$
neighborhoods of the shadow boundaries (see Fig.~\ref{fig:RegionsLayers}).
We overcome this difficulty by constructing approximation spaces improving
upon our approach in \cite{EcevitOzen16}, and better adapted to the changes in
frequency through use of a wavenumber dependent change of variables that
resolves the aforementioned boundary layers and that provides a smooth
transition from the shadow boundaries into the illuminated and shadow regions.
The resulting schemes, when compared with our algorithms in \cite{EcevitOzen16},
are easier to implement since the Galerkin spaces are represented as the
direct sum of a \emph{fixed} number approximation spaces rather than a number of those
that should \emph{increase} in proportion to $\log k$ as in \cite{EcevitOzen16}. Moreover, as we
explain, they display better approximation properties from a theoretical perspective
since they provide savings on the order of $\sqrt{\log k}$ to attain the same
accuracy. Perhaps more importantly, they yield significantly superior accuracy in the
shadow region as depicted through the numerical tests.

To describe our approximation spaces, we choose $\gamma$ to be the $L-$periodic
arc length parameterization of $\partial K$ in the counterclockwise orientation
with $\alpha \cdot \nu \left( \gamma(0) \right) = 1$. This ensures that if
$0 < t_{1} < t_{2} < L$ are the points corresponding to the \emph{shadow
boundaries}
\[
	\gamma \left( \left\{ t_1,t_2 \right\} \right)
	= \partial K^{SB},
\]
then the \emph{illuminated} and \emph{shadow regions} are given by (see
Fig.~\ref{fig:RegionsLayers}(a))
\begin{align*}
	\gamma \left( \left( t_1,t_2 \right) \right)
	= \partial K^{IL}
	\qquad
	\text{and}
	\qquad
	\gamma \left( \left( t_2,t_1+L \right) \right)
	= \partial K^{SR}.
\end{align*}

For $k >1$, we introduce two types of approximation spaces confined to the regions
depicted in Figure~\ref{fig:Regions}(a)-(b). In both cases, we define the \emph{illuminated
transition} and \emph{shadow transition} intervals as
\begin{align*}
	\itone & = [t_1 + \xi_1 k^{-1/3}, t_1 + \xi_1'] = [a_1,b_1], \\
	\ittwo & = [t_2 - \xi_2' , t_2 - \xi_2 k^{-1/3} ] = [a_2,b_2], \\
	\stone & = [t_1 - \zeta_1', t_1 - \zeta_1 k^{-1/3} ] = [a_3,b_3], \\
	\sttwo & = [t_2 + \zeta_2 k^{-1/3}, t_2 + \zeta_2'] = [a_4,b_4],
\end{align*}
and the \emph{shadow boundary} intervals as
\begin{align*}
	\sbone & = [t_1 - \zeta_1 k^{-1/3}, t_1 + \xi_1 k^{-1/3}] = [a_5,b_5], \\
	\sbtwo & = [t_2 - \xi_2 k^{-1/3}, t_2 + \zeta_2 k^{-1/3} ] = [a_6,b_6],
\end{align*}
where the parameters $\xi_j, \xi_j', \zeta_j,\zeta_j'>0$ ($j=1,2$) are chosen so that
\[
	t_1 + \xi_1
	\le t_1 + \xi_1'
	\overset{(A)}{\le} t_2 - \xi_2'
	\le t_2 - \xi_2,
\]
and
\[
	t_2 + \zeta_2
	\le t_2 + \zeta_2'
	\overset{(B)}{\le} L + t_1 - \zeta_1'
	\le L + t_1 - \zeta_1.
\]
\begin{figure}[hptb]
\begin{center}
\hbox{
	\hspace{0.5cm}
	\subfigure[Equalities in (A) and (B)]
	{\includegraphics*[height=1.45in,viewport=210 10 1575 790,clip]{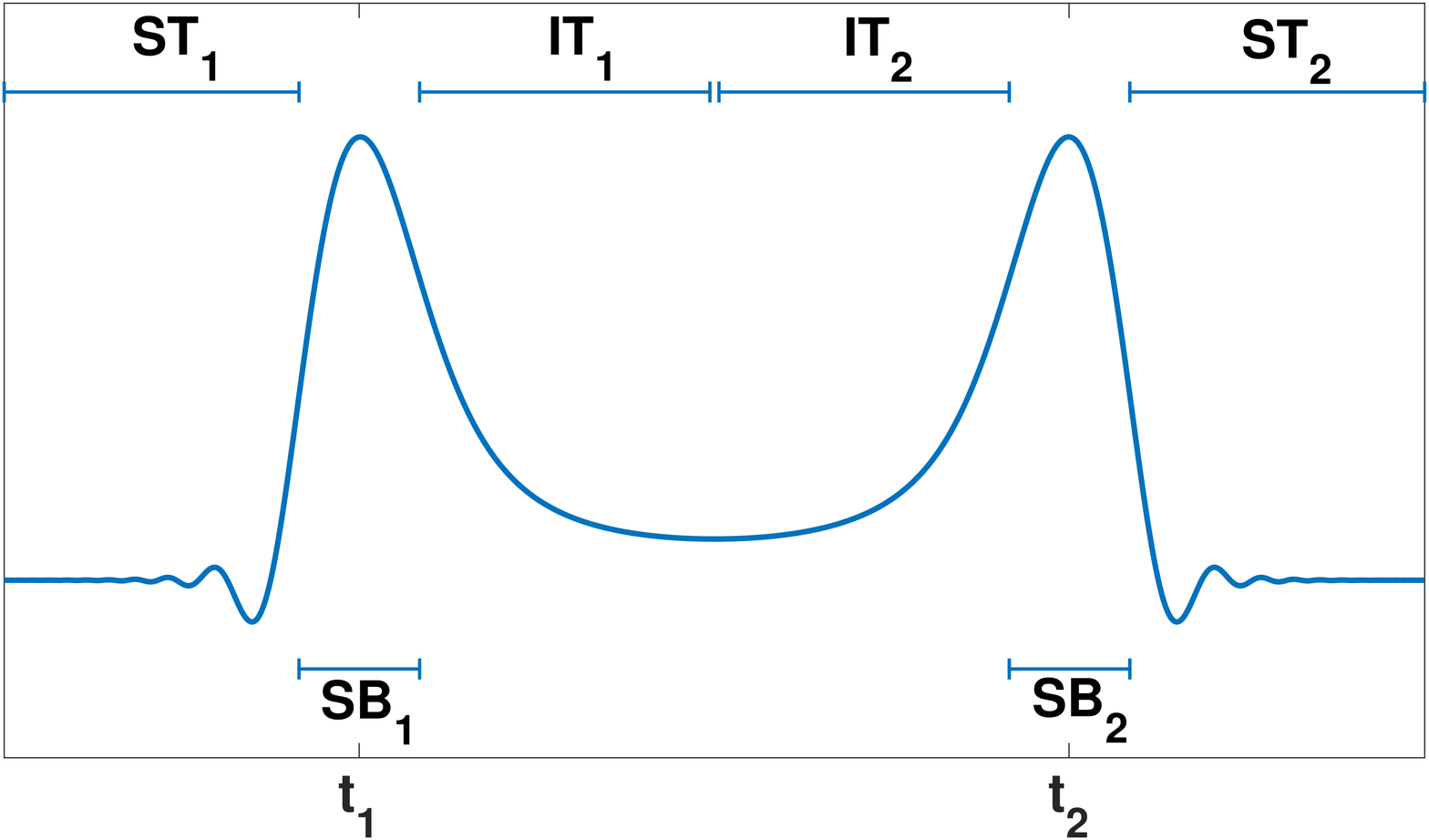}}
	\hspace{1cm}
	\subfigure[Strict inequalities in (A) and (B)]
	{\includegraphics*[height=1.45in,viewport=210 10 1575 790,clip]{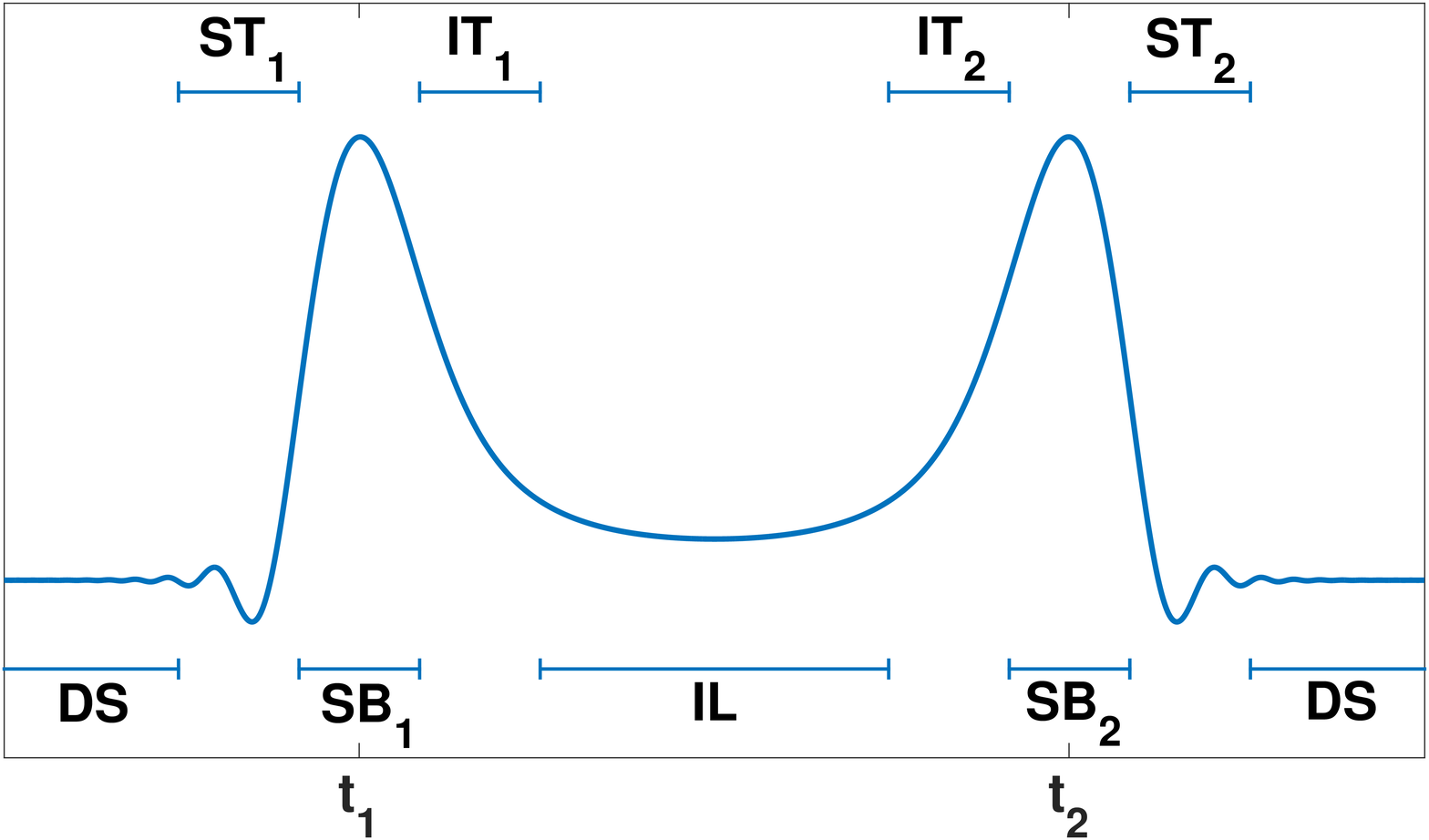}}
}
\end{center}
\caption{Regions on the boundary of the unit circle.}
\label{fig:Regions}
\end{figure}

Note specifically that the regions in Figure~\ref{fig:Regions}(a) correspond to equalities
in $(A)$ and $(B)$, and those in Figure~\ref{fig:Regions}(b) to strict inequalities. 
In the latter case, we define the \emph{illuminated} and \emph{deep shadow} intervals as
\begin{align*}
	\il & = [t_1 + \xi_1', t_2 - \xi_2'] = [a_7,b_7], \\
	\ds & = [	t_2 + \zeta_2', L + t_1 - \zeta_1'] = [a_8,b_8].
\end{align*}

With these choices, given $\mathbf{d} = \left( d_1, \ldots, d_J \right) \in \mathbb{Z}_+^{J}$
(with $J = 6$ and $J = 8$ for Figure~\ref{fig:Regions}(a) and (b) respectively), we define
our $(|\mathbf{d}|+J)-$dimensional \emph{Galerkin approximation spaces based on algebraic
polynomials and frequency dependent changes of variables} as
\begin{align*}
	\mathcal{A}_{\mathbf{d}}^{\mathcal{C}}
	= \bigoplus_{j=1}^{J}
	\charfunc_{\mathcal{I}_j} \ e^{i k \, \alpha \cdot \gamma} \ \mathcal{A}_{d_j}^{\mathcal{C}}.
\end{align*}
Here $\charfunc_{\mathcal{I}_j}$ is the characteristic function of $\mathcal{I}_j = [a_j,b_j]$, and
\begin{align*}
	\mathcal{A}_{d_j}^{\mathcal{C}} =
	\left\{
		\begin{array}{cl}
			\Pol_{d_j} \circ \phi ^{-1} , & \mbox{if } \mathcal{I}_j \mbox{ is a transition region}, \\ [0.5 em]
			\Pol_{d_j} , & \mbox{otherwise},
		\end{array}
	\right.	
\end{align*}
where $\Pol_{d_j}$ is the vector space of algebraic polynomials of degree at most $d_j$,
and $\phi$ is the change of variables on the transition intervals given explicitly as
\begin{align*}
	\phi \left( s \right)  = \left\{
		\begin{array}{ll} 
			t_1 + \varphi \left( s \right) k^{\psi \left( s \right)}, & s \in \itone,\\ [0.3 em]
			t_2 - \varphi \left( s \right) k^{\psi \left( s \right)}, & s \in \ittwo,\\ [0.3 em]
			t_1 - \varphi \left( s \right) k^{\psi \left( s \right)}, & s \in \stone,\\ [0.3 em]
			t_2 + \varphi \left( s \right) k^{\psi \left( s \right)}, & s \in \sttwo, 
		\end{array}
	\right .
\end{align*}
wherein $\varphi$ is the affine map
\begin{align*}
	\varphi (s)  = \left\{
		\begin{array}{ll}
			\xi_1 + \left( \xi_1'-\xi_1 \right) \dfrac{s-a_1}{b_1-a_1}, & s \in \itone,\\ [1 em]
		 	\xi_2' + \left( \xi_2-\xi_2' \right) \dfrac{s-a_2}{b_2-a_2}, & s \in \ittwo,\\ [1 em]
		 	\zeta_1' + \left( \zeta_1-\zeta_1' \right) \dfrac{s-a_3}{b_3-a_3}, & s \in \stone,\\ [1 em]
		 	\zeta_2 + \left( \zeta_2'-\zeta_2 \right) \dfrac{s-a_4}{b_4-a_4}, & s \in \sttwo,
		\end{array} 
	\right .
\end{align*}	 
and $\psi$ is the linear map
\begin{align*}
	\psi (s)  = -\dfrac{1}{3}
	\left\{
		\begin{array}{ll} 
			\dfrac{b_1-s}{b_1-a_1}, & s \in \itone,\\ [1 em]
		 	\dfrac{s-a_2}{b_2-a_2}, & s \in \ittwo,\\ [1 em]
		 	\dfrac{s-a_3}{b_3-a_3}, & s \in \stone,\\ [1 em]
		 	\dfrac{b_4-s}{b_4-a_4}, & s \in \sttwo.
		\end{array} 
	\right .
\end{align*}
The change of variables $\phi$ is constructed so that,
for $j=1,2,3,4$, the map $\phi : [a_j,b_j] \to [a_j,b_j]$ is an orientation
preserving diffeomorphism and the exponent $\psi$ of $k$ increases linearly from $-1/3$
to $0$ as we move away from the shadow boundaries into the illuminated or shadow regions.
While on the one hand this choice guarantees that the degrees of freedom assigned to the
$\mathcal{O}(k^{-1/3})$ neighborhoods of the shadow boundaries remains fixed with increasing
wave-number $k$, and on the other hand it also ensures that the approximation spaces are
perfectly adapted to the asymptotic behaviour of the solution.

\begin{remark}
Note that, by construction, $\gamma ( \cup_{j=1}^{J} \mathcal{I}_j ) = \partial K$ and the
intervals $\mathcal{I}_j$ intersect either trivially or only at an end point. Therefore we can clearly
identify $L^{2}\left( \partial K \right)$ and $L^2 ( \cup_{j=1}^{J} \mathcal{I}_j )$ through the
($L-$periodic arc length) parametrization $\gamma$. This will be the convention we shall
follow without any further reference in the rest of the paper.
\end{remark}

With the above definitions, the Galerkin formulation of problem \eqref{eq:CFIE} is
equivalent to finding the unique $\hat{\eta} \in \mathcal{A}_{\mathbf{d}}^{\mathcal{C}}$ such that
\begin{equation} \label{eq:galerkin-formulation-chvar}
	B_k(\hat{\mu},\hat{\eta})
	= F_k(\hat{\mu}),
	\quad
	\text{for all } \hat{\mu} \in \mathcal{A}_{\mathbf{d}}^{\mathcal{C}}.
\end{equation}
While the following theorem constitutes the main result of the paper and provides the
approximation properties of the solution of equation \eqref{eq:galerkin-formulation-chvar},
its proof is deferred to the next section. Hereafter we write $A \lesssim_{a,b,\ldots}B$ to
mean $0 \le A \le c B$ for a positive constant $c$ that depends only on $a,b,\ldots$.

\begin{theorem} \label{thm:changeofvariables}
Given $k_0 >1$ and $k \ge k_0$, suppose that the sesquilinear form $B_k$ in (\ref{eq:sesqui})
associated with the integral operator $\mathcal{R}_k$ in (\ref{eq:CFIE}) is continuous with a
continuity constant $C_k$ and coercive with a coercivity constant $c_k$. Then, for all
$n_{j} \in \{ 0,\ldots,d_{j} + 1 \}$ $(j=1,\ldots,J)$, we have
\[
	\Vert \eta - \hat{\eta} \Vert_{L^2(\partial K)}
	\lesssim_{n_1,\ldots,n_J,k_0}
	\dfrac{C_k}{c_k} \, k
	\sum_{j=1}^{J}
	\dfrac{E(k,j)}{\left( d_j \right)^{n_j}}
\]
for the Galerkin solution $\hat{\eta}$ to (\ref{eq:galerkin-formulation-chvar}) where
\[
	E(k,j) =
	\left\{
		\begin{array}{ll}
			\left( \log k \right)^{n_j+1/2}, & j =1,2,3,4 \ (\text{transition regions}), \\[0.5em]
			k^{-1/6},& j =5,6 \, \qquad (\text{shadow boundaries});
		\end{array}
	\right. 
\]
if $J =8$, then
\[
	E(k,j) = 1,
	\qquad
	j = 7,8,
	\quad
	(\text{illuminated and shadow regions}).
\]

\end{theorem}

Since $\eta$ has the same asymptotic order with $k$ as $k \to \infty$ (see e.g. \cite{MelroseTaylor85}),
if we assign the same polynomial degree to each interval $\mathcal{I}_j$, we obtain the following estimate for
the relative error.

\begin{corollary} \label{cor:chvar}
Under the assumptions of Theorem \ref{thm:changeofvariables}, if the same polynomial degree
$d = d_1 = \ldots = d_{J}$ is used on each interval, then for all $n \in \{0,\ldots, d+1\}$, there holds
\begin{equation} \label{eq:newalgerror}
	\dfrac{\Vert \eta - \hat{\eta} \Vert_{_{L^2(\partial K)}}}{\Vert \eta \Vert_{_{L^2(\partial K)}}}
	\lesssim_{n,k_0}
	\dfrac{C_k}{c_k} \dfrac{\left( \log k \right)^{n+1/2}}{d^n} 
\end{equation}
for the Galerkin solution $\hat{\eta}$ to (\ref{eq:galerkin-formulation-chvar}).
\end{corollary}

Theorem~\ref{thm:changeofvariables} and Corollary \ref{cor:chvar} display the improved convergence characteristics
of the Galerkin approximation spaces based on changes of variables when compared with our approach in \cite{EcevitOzen16}
wherein we have treated the transition regions in a different way (specifically, rather than utilizing a change of variables,
we have divided the transition regions into an optimal number of subregions depending on the underlying frequency and
used approximation spaces in the form of the plane-wave weighted by polynomials). Moreover, as will be depicted in the
numerical tests, the Galerkin approximation spaces based on changes of variables display a significant improvement
over those in \cite{EcevitOzen16} in terms of the accuracy of solutions in the shadow regions.

In order to allow for a direct comparison, here we present a more flexible version of our algorithm in \cite{EcevitOzen16},
that is better suited for different geometries, together with the associated convergence results. To this end, given
$m \in \N$, $0 \leq \epsilon_m < \epsilon_{m-1} < \cdots < \epsilon_1 < 1/3$, and constants
$\xi_1, \xi_2, \zeta_1,\zeta_2 >0$ with $t_1-\xi_1 < t_2-\xi_2$ and $t_2 + \zeta_2 < L+ t_1-\zeta_1$,
we define the associated \emph{illuminated transition} and \emph{shadow transition} intervals as
\begin{align*}
	\itone & = [t_1 + \xi_1 k^{-1/3+\epsilon_{m}}, t_1 + \xi_1 k^{-1/3+\epsilon_1}],
	\\
	\ittwo & = [t_2 - \xi_2 k^{-1/3+\epsilon_{1}}, t_2 - \xi_2 k^{-1/3+\epsilon_{m}}],
	\\
	\stone & = [t_1 - \zeta_1 k^{-1/3+\epsilon_1} , t_1 - \zeta_1 k^{-1/3+\epsilon_{m}}],
	\\
	\sttwo & = [t_2 +  \zeta_2 k^{-1/3+\epsilon_{m}} , t_2 + \zeta_2 k^{-1/3+\epsilon_{1}}],
\end{align*}
and, rather than utilizing a change of variables in the transition intervals, we divide each one into
$m-1$ subintervals by setting
\begin{align*}
	\itone^j & = [t_1 + \xi_1 k^{-1/3+\epsilon_{j+1}}, t_1 + \xi_1 k^{-1/3+\epsilon_j}],
	\\
	\ittwo^j & = [t_2 - \xi_2 k^{-1/3+\epsilon_{j}}, t_2 - \xi_2 k^{-1/3+\epsilon_{j+1}}],
	\\
	\stone^j & = [t_1 - \zeta_1 k^{-1/3+\epsilon_j} , t_1 - \zeta_1 k^{-1/3+\epsilon_{j+1}}],
	\\
	\sttwo^j & = [t_2 +  \zeta_2 k^{-1/3+\epsilon_{j+1}} , t_2 + \zeta_2 k^{-1/3+\epsilon_{j}}],
\end{align*}
for  $j=1,\ldots , m-1$. In addition to these $4m-4$ transition intervals, we further define
the \emph{shadow boundary} intervals as
\begin{align*}
	\sbone & = [t_1 - \zeta_1 k^{-1/3 +\epsilon_m}, t_1 + \xi_1 k^{-1/3+\epsilon_m}],
	\\
	\sbtwo & = [t_2 - \xi_1 k^{-1/3 +\epsilon_m}, t_2 + \zeta_2 k^{-1/3+\epsilon_m}],
\end{align*}
and the \emph{illuminated region} and \emph{shadow region} intervals as
\begin{align*}
	\il & = [t_1 + \xi_1k^{-1/3+\epsilon_1}, t_2 - \xi_2k^{-1/3+\epsilon_1}],
	\\
	\ds & = [t_2 + \zeta_2 k^{-1/3+\epsilon_1}, L+t_1-\zeta_1 k^{-1/3+\epsilon_1}].
\end{align*}
These give rise to a total of $4m$ intervals which we shall denote as $\mathcal{I}_j$ $(j=1,\ldots,4m)$.
Reasoning as before, we identify $L^{2} \left( \partial K \right)$ with $L^2 ( \cup_{j=1}^{4m} \mathcal{I}_j )$.
Given $\mathbf{d} = \left(d_1,\ldots,d_{4m} \right) \in \mathbb{Z}_+^{4m}$, we now define
the $( \left| \mathbf{d} \right| + 4m)-$dimensional \emph{Galerkin approximation space based
on algebraic polynomials} as
\begin{equation} \label{eq:GalerkinPoly}
	\mathcal{A}_\mathbf{d}
	= \bigoplus_{j=1}^{4m} \,
	\charfunc_{\mathcal{I}_j} \
	e^{i k \, \alpha \cdot \gamma} \
	\Pol_{d_j}.
\end{equation}
The associated Galerkin formulation of (\ref{eq:CFIE}) is to find the unique
$\hat{\eta} \in \mathcal{A}_{\mathbf{d}}$ such that
\begin{equation} \label{eq:galerkin-formulation}
	B_k(\hat{\mu},\hat{\eta})
	= F_k(\hat{\mu}),
	\quad
	\text{for all } \hat{\mu} \in \mathcal{A}_{\mathbf{d}}.
\end{equation}
Incidentally, the Galerkin approximation spaces defined by equation
\eqref{eq:GalerkinPoly} provide a more flexible version of those in \cite{EcevitOzen16}
since here we allow for different $\xi$ and $\zeta$ values rather than the 
values $\xi_1 = \xi_2$ and $\zeta_1=\zeta_2$ we used in \cite{EcevitOzen16}. This
clearly renders the new approximation spaces better adapted to different geometries.

\begin{theorem} \label{thm:ecevitozen}
Suppose that $k$ is sufficiently large and the sesquilinear form $B_k$ in (\ref{eq:sesqui}) associated with the integral
operator $\mathcal{R}_k$ in (\ref{eq:CFIE}) is continuous with a continuity constant $C_k$ and coercive with a coercivity
constant $c_k$. Then, for all $n_{j} \in \{ 0,\ldots,d_{j} + 1 \}$ $(j=1,\ldots,4m)$, we have
\[
	\Vert \eta - \hat{\eta} \Vert_{L^{2}(\partial K)}
	\lesssim_{n_{1},\ldots,n_{4m}}
	\dfrac{C_k}{c_k} \, k \,
	\sum_{j = 1}^{4m} \dfrac{1+E(k,j)}{\left( d_{j} \right)^{n_{j}}}
\]
for the Galerkin solution $\hat{\eta}$ to (\ref{eq:galerkin-formulation}) where
\[
	E(k,j) =
	\left\{
		\begin{array}{ll}
			k^{-(1+3\epsilon_{r+1})/2} \left( k^{(\epsilon_{r}-\epsilon_{r+1})/2} \right)^{n_{j}}\!\!\!\!\!,
			& j=1,\ldots,4m-4, \, \quad (\text{transition regions}),
			\\ [0.5 em]
			k^{-1/2} \left( k^{\epsilon_{m}} \right)^{n_{j}},
			& j=4m-3,4m-2, \ (\text{shadow boundaries}),
			\\ [0.5 em]
			k^{-(1+3\epsilon_{1})/2} \left( k^{(1/3-\epsilon_{1})/2} \right)^{n_{j}},
			& j=4m-1,4m, \qquad (\text{illuminated \& shadow reg.}).
		\end{array}
	\right.
\]
\end{theorem}
The proof of Theorem \ref{thm:ecevitozen} is similar to that of Theorem $1$ in \cite{EcevitOzen16} and
is therefore skipped. On the other hand, exactly as Theorem 1 therein implies Corollary 1 in
\cite{EcevitOzen16}, Theorem \ref{thm:ecevitozen} above yields the following result.

\begin{corollary}\label{corollary:algebraic1}
Under the assumptions of Theorem \ref{thm:ecevitozen}, if $\epsilon_j$ are chosen as
\[
	\epsilon_{j} = \dfrac{1}{3} \, \dfrac{2m-2j+1}{2m+1},
	\qquad
	 j = 1,\ldots,m, 
\]
and the same polynomial degree $d = d_1 = \ldots = d_{4m}$ is used on each interval,
then for all $n \in \{ 0,\ldots,d+1 \}$, there holds
\begin{equation} \label{eq:oldalgerror}
	\dfrac{\Vert \eta - \hat{\eta} \Vert_{_{L^2(\partial K)}}}{\Vert \eta \Vert_{_{L^2(\partial K)}}}
	\lesssim_{n} \dfrac{C_k}{c_k} \, m \, \dfrac{1+ k^{-\frac{1}{2}} \left( k^{\frac{1}{6m+3}} \right)^{n}}{d^{n}}
\end{equation}
for the Galerkin solution $\hat{\eta}$ to (\ref{eq:galerkin-formulation}).
\end{corollary}

\begin{remark} Note specifically that if $m$ $($and thus the total number of degrees of freedom$)$
increases in parallel with $\log k$, then $k^{\frac{1}{6m+3}} = \exp(\log k/(6m+3))$ is bounded independently
of $k$, and the estimate \eqref{eq:oldalgerror} takes on the form
\[
	\dfrac{\Vert \eta - \hat{\eta} \Vert_{_{L^2(\partial K)}}}{\Vert \eta \Vert_{_{L^2(\partial K)}}}
	\lesssim_{n} \dfrac{C_k}{c_k} \, \dfrac{\log k}{d^{n}}.
\]
As for the estimate \eqref{eq:newalgerror} in Corollary \ref{cor:chvar} relating to the Galerkin schemes
based on changes of variables, if the common local polynomial degree $d$ is proportional to $\log k$ $($say $d \approx d_0 \, \log k$, with $d_0$
independent of $k$, so that the total number of degrees of freedom increases as $\log k$$)$, then the estimate 
in Corollary reduces to
\[
	\dfrac{\Vert \eta - \hat{\eta} \Vert_{_{L^2(\partial K)}}}{\Vert \eta \Vert_{_{L^2(\partial K)}}}
	\lesssim_{n,k_0} \dfrac{C_k}{c_k} \, \dfrac{\sqrt{\log k}}{d_0^n}.
\]
This shows that the Galerkin schemes based on change of variables display better approximation properties
when compared with the frequency-adapted Galerkin schemes in \cite{EcevitOzen16}. Furthermore, since $C_k/c_k = \mathcal{O}(k^{\delta})$
with $\delta = 1/2$ for SCIE and $\delta = 1/3$ for CFIE when the obstacle $K$ is strictly convex (see \cite{BaskinEtAl16}),
given $\varepsilon >0$, if $d_0$ grows in parallel with $k^{\varepsilon}$,
for sufficiently large $n_0$, there holds
\[
	\dfrac{\Vert \eta - \hat{\eta} \Vert_{_{L^2(\partial K)}}}{\Vert \eta \Vert_{_{L^2(\partial K)}}}
	\lesssim_{n,k_0} \dfrac{1}{d_0^{n-n_0}}
\]
for all $n \ge n_0$. This shows that, for any $\epsilon >0$, increasing the total number of degrees of freedom associated
with the Galerkin schemes based on change of variables as $\mathcal{O}(k^{\epsilon})$ is sufficient to obtain any
prescribed accuracy independent of frequency.  
\end{remark}

\section{Error analysis}
\label{sec:4}

In this section we present the proof of Theorem \ref{thm:changeofvariables}. In light of inequality \eqref{eq:Cea},
it is sufficient to estimate
\[
	\inf_{\hat{\mu} \in \mathcal{A}_{\mathbf{d}}^{\mathcal{C}}} \Vert \eta- \hat{\mu} \Vert_{L^2(\partial K)}.
\]
To this end, we make use of the following classical result from approximation theory.
\begin{theorem}[Best approximation by algebraic polynomials \cite{Schwab98}]
\label{thm:pae}
Given an interval $\mathcal{I} = (a,b)$ and $n \in \mathbb{Z}_+$, introduce the semi--norms $($for suitable $f$$)$ by
\begin{equation} \label{eq:seminorm}
	\left| f \right|_{n,\mathcal{I}} = 
	\left[
		\int_{a}^{b}
		\left| D^{n} f(s) \right|^{2}
		\left( s-a \right)^{n}
		\left( b-s \right)^{n}
		ds
	\right]^{1/2} .
\end{equation}
Then, for all $n \in \{ 0, \ldots, d+1\}$, there holds
\begin{equation*}
	\inf_{p \in \mathbb{P}_d}
	\Vert f-p \Vert
	\lesssim_{n} \left| f \right|_{n,\mathcal{I}} \, d^{-n}.
\end{equation*}
\end{theorem}
Use of Theorem \ref{thm:pae}, in turn, requires the knowledge of derivative estimates of $\eta^{\rm slow}$
which we present next.
\begin{theorem} \label{thm:etaslowder}
Given $k_0 > 0$, there holds
\begin{equation*}
	\left| D_{s}^{n} \eta^{\rm slow}(s,k) \right|
	\lesssim_{n,k_0}
	k + \sum\limits_{m=4}^{n+2} \left( k^{-1/3} + \left| w(s) \right| \right)^{-m}
\end{equation*}
for all $n \in \mathbb{Z}_+$ and all $ k \ge k_0$. Here $w(s) = (s-t_{1})(t_{2}-s)$.
\end{theorem}
\begin{proof}
The same estimate is shown to hold for all sufficiently large $k$ in \cite{DominguezEtAl07}.
Since $D_{s}^{n} \eta^{\rm slow}(s,k)$ depends continuously on $s$ and $k$, the result follows. 
\end{proof}

We continue with the derivation of estimates on the derivatives of the change of variables $\phi$
on the transition interval $\mathcal{I}_j$ $(j=1,2,3,4)$.
\begin{proposition} \label{prop:phider}
Given $k_0 > 1$, there holds
\[
	\left| D^n_s \phi \right|
	\lesssim_{n,k_0} \left( \log k \right)^{n}  k^{\psi}
	\quad
	\text{on } \mathcal{I}_j \quad (j=1,2,3,4),
\]
for all $n \in \mathbb{N}$ and all $k \ge k_0$.
\end{proposition}
\begin{proof}
Since the proof is similar for $j=1,2,3,4$, we concentrate on the case $j=1$.
Now since $\varphi'' = \psi'' = 0$, direct computations entail
\begin{equation} \label{eq:phider}
	D^n_s \phi
	= \sum_{j=0}^{n} \binom{n}{j} \, D^{n-j}_s \varphi \ D^{j}_s k^{\psi}
	= \varphi \ D^{n}_s k^{\psi} + D^{1}_s \varphi \ D^{n-1}_s k^{\psi},
	\qquad
	n \ge 1,
\end{equation}
and
\begin{equation} \label{eq:kpsider}
	D^n_s k^{\psi}
	= \left( D^{1}_s \psi \right)^n \left( \log k \right)^n k^{\psi},
	\qquad
	n \ge 0.
\end{equation}
Using \eqref{eq:kpsider} in \eqref{eq:phider}, we obtain
\begin{equation} \label{eq:phiderfinal}
	D^n_s \phi
	= \left( \varphi \ D^{1}_s \psi \ \log k + D^{1}_s \varphi \right) \left( D^{1}_s \psi \right)^{n-1} \left( \log k \right)^{n-1}  k^{\psi},
	\qquad
	n \ge 1.
\end{equation}
Since, for $k \ge k_0 > 1$,
\[
	\dfrac{1}{\left| b_1-a_1 \right|}
	= \dfrac{1}{\xi_1' - \xi_1k^{-1/3}}
	\le \dfrac{1}{\xi_1' - \xi_1k_0^{-1/3}} 
	\lesssim_{k_0} 1
\]
it follows for $s \in \mathcal{I}_1 = I_{IT_1}$ that
\[
	\left| D^{1}_s \psi \right|
	= \dfrac{1}{3} \, \dfrac{1}{b_1-a_1}
	\lesssim_{k_0} 1
	\qquad
	\text{and}
	\qquad
	\left| D^{1}_s \varphi \right|
	= \dfrac{\xi_1' - \xi_1}{b_1-a_1}
	\lesssim_{k_0} 1
\]
and, clearly, $\xi_1 \le \varphi \le \xi_1'$. Use of these inequalities in \eqref{eq:phiderfinal} yields
the desired result.
\end{proof}

Next we combine Theorem \ref{thm:etaslowder} and Proposition \ref{prop:phider} to derive estimates
on the derivatives of the composition $\eta^{\rm slow} \circ \phi$.
\begin{proposition} \label{prop:dercomp}
Given $k_0 > 0$, there holds
\begin{align*}
	\left| D_s^n (\eta^{slow} \circ \phi) \right|
	\lesssim_{n,k_0} k \left( \log k \right)^n
	\quad
	\text{on }
	\mathcal{I}_j
	\quad
	\left( j=1,2,3,4 \right),
\end{align*}
for all $n \in \mathbb{N}$ and all $k \ge k_0$.
\end{proposition}
\begin{proof}
We fix $k \ge k_0 > 1$ and the interval $\mathcal{I}_j$ ($j=1,\ldots, 4$).
Fa\'{a} Di Bruno's formula for the derivatives of a composition states
\[
	D^{n} \left( f \circ g \right) \left( t \right)
	= \sum_{\left\{ m_{\ell} \right\}} (D^{m} f) (g(t))
	\prod_{\ell=1}^{n} \dfrac{\ell}{m_{\ell}!} \left( \dfrac{D^{\ell}g(t)}{\ell!} \right)^{m_{\ell}}
\]
where the summation is over all $m_{\ell} \in \mathbb{Z}_+$ with $n = \sum_{\ell=1}^{n} \ell m_{\ell}$;
here $m = \sum_{\ell=1}^{n} m_{\ell}$. This yields
\[
	\left| D^n_s \left( \eta^{\rm slow} \circ \phi \right) \right|
	\lesssim_{n} \sum_{\left\{ m_{\ell} \right\}} \left| (D^{m}_s \eta^{\rm slow}) (\phi) \right|
	\prod_{\ell=1}^{n} \left| D^{\ell}_s \phi \right|^{m_{\ell}}
\]
so that an appeal to Proposition \ref{prop:phider} entails
\[
	\left| D^n_s \left( \eta^{\rm slow} \circ \phi \right) \right|
	\lesssim_{n,k_{0}} \sum_{\left\{ m_{\ell} \right\}} \left| (D^{m}_s \eta^{\rm slow}) (\phi) \right|
	\prod_{\ell=1}^{n} \left( \left( \log k \right)^{\ell} k^{\psi} \right)^{m_{\ell}}.
\]
Since $n = \sum_{\ell=1}^{n} \ell m_{\ell}$ and $m = \sum_{\ell=1}^{n} m_{\ell}$, we therefore obtain
\begin{align*}
	\left| D^n_s \left( \eta^{\rm slow} \circ \phi \right) \right|
	& \lesssim_{n,k_{0}} \left( \log k \right)^n
	\sum_{\left\{ m_{\ell} \right\}} \left| (D^{m}_s \eta^{\rm slow}) (\phi) \right|
	k^{m\psi}
	\\
	& \lesssim_{n,k_{0}} \left( \log k \right)^n
	\sum_{m=0}^{n} \left| (D^{m}_s \eta^{\rm slow}) (\phi) \right|
	k^{m\psi}.
\end{align*}
It is hence sufficient to show, for $m \in \mathbb{Z}_+$, that
\[
	\left| (D^{m}_s \eta^{\rm slow}) (\phi) \right|
	k^{m\psi}
	\lesssim_{m,k_{0}} k.
\]
To this end, we note that if $0 \le \ell \le m$, then
\begin{align*}
	\left( k^{-1/3} + \left| \omega (\phi) \right| \right)^{-\ell}
	& = \left( k^{-1/3} + \left| \omega (\phi) \right| \right)^{-m}
	\left( k^{-1/3} + \left| \omega (\phi) \right| \right)^{ m-\ell}
	\\
	& \leq \left( k^{-1/3} + \left| \omega (\phi) \right| \right)^{-m}
	\left( k_0^{-1/3} + L^2 \right)^{m-\ell}
	\\
	&\lesssim_{m,k_0}  \left( k^{-1/3} + \left| \omega (\phi) \right| \right)^{-m}
\end{align*}
so that an appeal to Theorem \ref{thm:etaslowder} yields
\begin{align*}
	\left| D_{s}^{m} \eta^{\rm slow}(\phi) \right| k^{m\psi}
	& \lesssim_{m,k_0}
	\left[ k + \sum_{\ell=4}^{m+2} \left( k^{-1/3} + \left| \omega(\phi) \right| \right)^{-\ell} \right]
	k^{m\psi}
	\\
	& \lesssim_{m,k_0}
	\left[ k + \left( k^{-1/3} + \left| \omega(\phi) \right| \right)^{-\left( m+2 \right)} \right]
	k^{m\psi}
	\\
	& \lesssim_{m,k_0}
	k + \left( \dfrac{k^{\psi}}{k^{-1/3} + \left| \omega(\phi) \right|} \right)^{m}
	\left( k^{-1/3} + \left| \omega(\phi) \right| \right)^{-2}
	\\
	& \lesssim_{m,k_0}
	k + \left( \dfrac{k^{\psi}}{\left| \omega(\phi) \right|} \right)^{m}
	k^{2/3}
\end{align*}
where, in the third inequality, we used that $\psi \le 0$. Thus it is now enough to show that the quotient
$k^{\psi}/\left| \omega(\phi) \right|$ is bounded by a constant independent of $k$. This estimation
is similar on each of the transition intervals $\mathcal{I}_j$ ($j=1,2,3,4$) and we focus on $\mathcal{I}_1$.
Indeed, on $\mathcal{I}_1=I_{IT_1}$, we have $\phi - t_1 = \varphi k^{\psi}$, $\varphi \ge \xi_1 > 0$ and
$t_2 - \phi \ge t_2 - \left( t_1 + \xi_1' \right) \ge \xi_2' > 0$ so that 
\begin{align*}
	\dfrac{k^{\psi}}{\left| \omega(\phi) \right|}
	= \dfrac{k^{\psi}}{\left( \phi -t_1 \right) \left( t_2 - \phi \right)}
	= \dfrac{k^{\psi}}{\varphi k^{\psi} \left( t_2 - \phi \right)}
	\le \dfrac{1}{\xi_1 \, \xi_2'}.
\end{align*}
This finishes the proof.
\end{proof}

Next we estimate the semi-norms \eqref{eq:seminorm} for the composition $\eta^{\rm slow} \circ \phi$
on the transition intervals $\mathcal{I}_j$ $(j=1,2,3,4)$.

\begin{corollary} \label{cor:snetaslowcirphi}
On the transition intervals $\mathcal{I}_j$ $(j=1,2,3,4)$, given $k_0 > 1$, there holds
\[
	\left| \eta^{\rm slow} \circ \phi \right|_{n,\mathcal{I}_j}
	\lesssim_{n,k_0} k \left( \log k \right)^n
\]
for all $n \in \mathbb{N}$ and all $k \ge k_0$.
\end{corollary}
\begin{proof}
On account of Proposition \ref{prop:dercomp}, we estimate for $j=1,2,3,4$
\begin{align*}
	\left| \eta^{\rm slow} \circ \phi \right|^2_{n,\mathcal{I}_j}
	& = \int_{a_j}^{b_j} \left| D^n_s\left( \eta^{\rm slow} \circ \phi \right) (s) \right|^2 \left( s - a_j \right)^n \left( b_j -s \right)^n ds
	\\
	& \lesssim_{n,k_0} k^2 \left( \log k \right)^{2n} \int_{a_j}^{b_j} \left( s - a_j \right)^n \left( b_j -s \right)^n ds
	\\
	& \lesssim_{n,k_0} k^2 \left( \log k \right)^{2n}
\end{align*}
where we used that $0< b_j-a_j <L$. Thus the result.
\end{proof}

We are now ready to prove Theorem \ref{thm:changeofvariables}.

\begin{proof}(of Theorem \ref{thm:changeofvariables}):
While C\'{e}a's lemma (cf. inequality \eqref{eq:Cea}) entails
\begin{equation}
\label{eq:CeaCd}
	\Vert \eta -\hat{\eta} \Vert_{L^2 \left( \partial K \right)}
	\leq \dfrac{C_k}{c_k}
	\inf_{\hat{\mu} \in \mathcal{A}_{\mathbf{d}}^{\mathcal{C}}} \Vert \eta- \hat{\mu} \Vert_{L^2 \left( \partial K \right)}
\end{equation}
for the unique solution $\hat{\eta}$ of the Galerkin formulation \eqref{eq:galerkin-formulation-chvar},
as we identify $L^2 \left( \partial K \right)$ with $L^2 ( \cup_{j=1}^{J} \mathcal{I}_j )$ through
the $L-$periodic arc length parameterization $\gamma$, there holds
\[
	\Vert \eta- \hat{\mu} \Vert_{L^2(\partial K)}
	= \Vert \eta- \hat{\mu} \Vert_{L^2 \left( \cup_{j=1}^{J} \mathcal{I}_j \right)}
	\le \sum_{j=1}^{J} \Vert \eta- \hat{\mu} \Vert_{L^2 \left( \mathcal{I}_j \right)}
\]
for any $\hat{\mu} \in \mathcal{A}_{\mathbf{d}}^{\mathcal{C}}$.
Accordingly, the very definition of Galerkin approximation spaces
$\mathcal{A}_{\mathbf{d}}^{\mathcal{C}}$ entails
\begin{equation} \label{eq:expand}
	\inf_{\hat{\mu} \in \mathcal{A}_{\mathbf{d}}^{\mathcal{C}}} \Vert \eta- \hat{\mu} \Vert_{L^2(\partial K)}
	\le \sum_{j=1}^{4} \inf_{p \in \mathbb{P}_{d_j}} \Vert \eta^{\rm slow} - p \circ \phi^{-1} \Vert_{L^2(\mathcal{I}_j)}
	+ \sum_{j=5}^{J} \inf_{p \in \mathbb{P}_{d_j}} \Vert \eta^{\rm slow} - p \Vert_{L^2(\mathcal{I}_j)}.
\end{equation}
On the other hand, utilizing the change of variables $\phi$ on the transition intervals $\mathcal{I}_j$
$(j=1,2,3,4)$, for any $p \in \mathbb{P}_{d_j}$, we have
\begin{align*}
	\Vert \eta^{\rm slow} - p \circ \phi^{-1} \Vert^2_{L^2(\mathcal{I}_j)}
	& = \int_{a_j}^{b_j} \left| \left( \eta^{\rm slow} - p \circ \phi^{-1}\right)(s) \right|^2 ds
	\\
	& = \int_{a_j}^{b_j} \left| \left( \eta^{\rm slow} \circ \phi - p \right)(s) \right|^2 D^{1}_s\phi(s) \, ds
	\\
	& \lesssim_{k_0} \log k \ \Vert \eta^{\rm slow} \circ \phi - p \Vert^2_{L^2(\mathcal{I}_j)}
\end{align*}
where we used Proposition \ref{prop:phider} in conjunction with the fact that $k^{\psi} < 1$. Combining
this last estimate with \eqref{eq:CeaCd} and \eqref{eq:expand}, we deduce
\[
	\Vert \eta -\hat{\eta} \Vert_{L^2 \left( \partial K \right)}
	\lesssim_{k_0} \dfrac{C_k}{c_k}
	\left\{
		\sum_{j=1}^{4} \left( \log k \right)^{1/2} \inf_{p \in \mathbb{P}_{d_j}} \Vert \eta^{\rm slow} \circ \phi - p \Vert_{L^2(\mathcal{I}_j)}
		+ \sum_{j=5}^{J} \inf_{p \in \mathbb{P}_{d_j}} \Vert \eta^{\rm slow} - p \Vert_{L^2(\mathcal{I}_j)}
	\right\}
\]
and this, on account of Theorem \ref{thm:pae}, implies
\[
	\Vert \eta -\hat{\eta} \Vert_{L^2 \left( \partial K \right)}
	\lesssim_{n_1,\ldots,n_{J},k_{0}} \dfrac{C_k}{c_k}
	\left\{
		\sum_{j=1}^{4} \left( \log k \right)^{1/2} \left| \eta^{\rm slow} \circ \phi \right|_{n_j,\mathcal{I}_j} d_j^{-n_j}
		+ \sum_{j=5}^{J} \left| \eta^{\rm slow} \right|_{n_j,\mathcal{I}_j} d_j^{-n_j}
	\right\}.
\]
Therefore, to complete the proof, it suffices to show that
\begin{equation} \label{eq:esttobeproved14}
	\left| \eta^{\rm slow} \circ \phi \right|_{n_j,\mathcal{I}_j}
	\lesssim_{n_j,k_0} k \left( \log k \right)^{n_j},
	\qquad
	j = 1,2,3,4,
\end{equation}
and
\begin{equation} \label{eq:esttobeproved56}
	\left| \eta^{\rm slow} \right|_{n_j,\mathcal{I}_j}
	\lesssim_{n_j,k_0} k \, k^{-1/6},
	\qquad
	j = 5,6,
\end{equation}
and (if $J=8$)
\begin{equation} \label{eq:esttobeproved78}
	\left| \eta^{\rm slow} \right|_{n_j,\mathcal{I}_j}
	\lesssim_{n_j,k_0} k,
	\qquad
	j = 7,8.
\end{equation}
While the estimates in \eqref{eq:esttobeproved14} are given by Corollary \ref{cor:snetaslowcirphi},
for the shadow boundary intervals $\mathcal{I}_j$ ($j=5,6$) we use Theorem \ref{thm:etaslowder}
to deduce
\begin{equation*}
	\left| D_{s}^{n_j} \eta^{\rm slow}(s,k) \right|
	\lesssim_{n_j,k_0}
	k + \sum_{m=4}^{n_j+2} \left( k^{-1/3} + \left| w(s) \right| \right)^{-m}
	\lesssim_{n_j,k_0} k + k^{(n_j+2)/3};
\end{equation*}
this implies
\begin{align*}
	\left| \eta^{\rm slow} \right|_{n_j,\mathcal{I}_j}^2
	& = \int_{a_j}^{b_j} \left| D^{n_j}_s \eta^{\rm slow}(s) \right|^2 \left( s -a_j \right)^{n_j} \left( b_j-s \right)^{n_j} ds
	\\
	& \lesssim_{n_j,k_0} \left( k + k^{(n_j+2)/3} \right)^2 \left( b_j -a_j \right)^{2n_j+1}
	\\
	& \lesssim_{n_j,k_0} \left( k + k^{(n_j+2)/3} \right)^2 \left( k^{-1/3} \right)^{2n_j+1}
	\\
	& \lesssim_{n_j,k_0} \left( k \, k^{-1/6} \right)^2
\end{align*}
which justifies the estimates in \eqref{eq:esttobeproved56}.  
This completes the proof when $J =6$.
 
When $J = 8$, for the illuminated and shadow region intervals $\mathcal{I}_j$ ($j=7,8$),
we use Theorem \ref{thm:etaslowder} to estimate
\begin{equation*}
	\left| D_{s}^{n_j} \eta^{\rm slow}(s,k) \right|
	\lesssim_{n_j,k_0} k + \sum_{m=4}^{n_j+2} \left( k^{-1/3} + \left| w(s) \right| \right)^{-m}
	\lesssim_{n_j,k_0} k + \sum_{m=4}^{n_j+2} \left| w(s) \right|^{-m}
	\lesssim_{n_j,k_0} k
\end{equation*}
so that
\begin{align*}
	\left| \eta^{\rm slow} \right|_{n_j,\mathcal{I}_j}^2
	& = \int_{a_j}^{b_j} \left| D^{n_j}_s \eta^{\rm slow}(s) \right|^2 \left( s -a_j \right)^{n_j} \left( b_j-s \right)^{n_j} ds
	\\
	& \lesssim_{n_j,k_0} k^2 \left( b_j -a_j \right)^{2n_j+1}
	\\
	& \lesssim_{n_j,k_0} k^2
\end{align*}
which verifies \eqref{eq:esttobeproved78}. This finishes the proof.
\end{proof}

\section{Numerical tests}
\label{sec:5}

While the earlier algorithms
\cite{BrunoEtAl04,Giladi07,HuybrechsVandewalle07,AsheimHuybrechs14,DominguezEtAl07}
concerning smooth convex obstacles are either not supported with rigorous numerical analysis
\cite{BrunoEtAl04,Giladi07,HuybrechsVandewalle07,AsheimHuybrechs14}
and/or remain asymptotic
\cite{BrunoEtAl04,HuybrechsVandewalle07,DominguezEtAl07}
as they approximate the solutions by zero in the deep shadow regions, our recent
frequency-adapted Galerkin boundary element methods \cite{EcevitOzen16}
are supported with a fully rigorous analysis and they provide approximations in the deep shadow
region. In this section, we therefore present numerical tests exhibiting the performance of
\emph{Galerkin boundary element methods based on changes of variables} developed herein
in comparison with those in \cite{EcevitOzen16}. Indeed, as the tests demonstrate, the new
schemes attain the same numerical accuracy with a reduced number of degrees of freedom
and, most strikingly, they provide significantly improved approximations in the \emph{shadow regions}.

On a related note, just as we have based the \emph{frequency-adapted Galerkin boundary element methods}
in \cite{EcevitOzen16} on either algebraic polynomials or trigonometric polynomials,
the Galerkin approximation spaces developed herein can also
be based on trigonometric polynomials in which case the Galerkin approximation spaces take
on the form
\begin{align*}
	\mathcal{T}^{\mathcal{C}}_{\mathbf{d}}
	= \bigoplus_{j=1}^{J}
	\charfunc_{\mathcal{I}_j} \ e^{i k \, \alpha \cdot \gamma} \ \mathcal{T}_{d_j}^{\mathcal{C}}
\end{align*}
where each $d_j$ is even,
\begin{align*}
	\mathcal{T}_{d_j}^{\mathcal{C}} =
	\left\{
		\begin{array}{cl}
			\mathbb{T}_{d_j}(\mathcal{I}_j) \circ \phi ^{-1} , & \mbox{if } \mathcal{I}_j \mbox{ is a transition region}, \\ [0.5 em]
			\mathbb{T}_{d_j}(\mathcal{I}_j) , & \mbox{otherwise},
		\end{array}
	\right.	
\end{align*}
and, for a generic interval $\mathcal{I} = [a,b]$ and an even integer $d$, $\mathbb{T}_d(\mathcal{I})$
is the space of $(b-a)$-periodic trigonometric polynomials of degree at most $d$ on $\mathcal{I}$.
In this section, we also present numerical tests that display the improvements provided
by the approximation spaces $\mathcal{T}^{\mathcal{C}}_{\mathbf{d}}$ over their trigonometric
counterparts
\begin{align*}
	\mathcal{T}_{\mathbf{d}}
	= \bigoplus_{j=1}^{4m}
	\charfunc_{\mathcal{I}_j} \ e^{i k \, \alpha \cdot \gamma} \ \mathbb{T}_{d_j}(\mathcal{I}_j)
\end{align*}
in \cite{EcevitOzen16}. Indeed, as in \cite{EcevitOzen16}, here the intervals
$\mathcal{I}_j$ are chosen to overlap with their immediate neighbors but the size of the overlap
diminishes as $k \to \infty$. This requirement is related with the need to introduce a smooth partition
of unity confined to the various regions on the boundary of the scatterer for both theoretical and
practical reasons. For details of this construction, we refer to \cite{EcevitOzen16}. Incidentally,
the error analysis corresponding to the spaces $\mathcal{T}^{\mathcal{C}}_{\mathbf{d}}$ can be
carried out utilizing the techniques in \S\ref{sec:4} in conjunction with those in \cite[\S4.2]{EcevitOzen16}.

An important component of our algorithms relates to the choice of bases for the spaces of algebraic and trigonometric polynomials
and, in order to minimize the numerical instabilities arising from the use of high-degree polynomials,
on any generic interval $\mathcal{I} = [a,b]$, we use the bases $\{ \rho^r : r=0,\ldots,d \}$ and $\{ (\rho \circ \phi^{-1})^r : r=0,\ldots,d \}$
for $\mathbb{P}_{d}$ and $\Pol_d \circ \phi^{-1}$ respectively where $\rho$ is the affine
function that maps the interval $\mathcal{I}$ onto $[-1,1]$. Similarly, for even values of $d$,
we employ the bases $\{ \exp({i r \rho}) : r=-\frac{d}{2},\ldots,\frac{d}{2} \}$ and
$\{ \exp({i r (\rho \circ \phi^{-1}}) : r=-\frac{d}{2},\ldots,\frac{d}{2} \}$ for $\Tri_d(\mathcal{I})$ and
$\Tri_d \circ \phi^{-1}(\mathcal{I})$ where, this time, $\rho$ maps the interval $\mathcal{I}$ onto
$[0,2\pi]$.

In the same vein, the choice of the parameters $\xi,\xi',\zeta,\zeta'$ appearing in the definitions of
the intervals $I_{IT_j},I_{ST_j},I_{SB_j}$ ($j=1,2$) and $I_{IT},I_{DS}$ is of great importance since
a random choice may result in a loss of accuracy due to poor resolution of the boundary layers in
the solution. We therefore optimize these parameters for a small wave number through a simple
iterative procedure, and use these values for all larger wave numbers. For instance, when $J = 6$,
we take $\xi_1' = \xi_2'$, $\zeta_1'=\zeta_2'$ and initially require that $\alpha \cdot \gamma(\xi_1') = -1$
and $\alpha \cdot \gamma(\zeta_1') = 1$. We then take $\xi_1$ to be the mid-point between
$t_1$ and $\xi_1'$, and change it in small increments until the local error in $IT_1 \cup SB_1$
is minimized. We treat the triplet $(t_2,\xi_2,\xi_2')$ similarly. Finally, we fix $\xi_1$ and $\xi_2$,
and change $\xi_1'=\xi_2'$ in small increments until the error in
$IL \cup IT_1 \cup IT_2 \cup SB_1 \cup SB_2$ is minimized. The optimization of $\zeta$
parameters is realized similarly. The computed values are taken as the initial guess
for the next iterate, and the procedure is repeated with smaller increments, with step size half
that of the previous one, until the variations in the global error stabilizes. Following the
prescriptions in \cite[\S4.2]{EcevitOzen16}, then we introduce a smooth partition of unity
confined to the regions $I_{IT_j},I_{ST_j},I_{SB_j}$ ($j=1,2$) and $I_{IT},I_{DS}$, and optimize
the shapes of hat functions therein using a similar iterative procedure (see the left-most panes
in Figures \ref{fig:circle_mono_vs_chanvar}--\ref{fig:kite}). For further details, we refer to
\cite{Eruslu15}.

As for the choice of the integral equation, as we mentioned, the developments central to this paper
are independent of the integral equation used. However, in order to allow a simple performance comparison
with the aforementioned algorithms, we base our numerical implementations on the CFIE
wherein the integral operator and the right-hand side are given by
\begin{equation*}
	\mathcal{R}_k = \dfrac{1}{2} \, I + \mathcal{D} - ik \mathcal{S}
	\quad
	\text{and}
	\quad
	f_k = \dfrac{\partial u^{\rm inc}}{\partial \nu}  - ik u^{\rm inc}.
\end{equation*}
Here $\mathcal{S}$ is the acoustic single-layer integral operator and $\mathcal{D}$ is its
normal derivative, and they are defined by
\[	\mathcal{S} \eta(x) = \int_{\partial K} \Phi(x,y) \, \eta(y) \, ds(y),
	\qquad
	x \in \partial K,
\]
\[
	\mathcal{D} \eta(x) = \int_{\partial K} \dfrac{ \partial \Phi(x,y)}{\partial \nu(x)} \, \eta(y) \, ds(y),
	\qquad
	x \in \partial K,
\]
where $\nu(x)$ is the outward unit normal to ${\partial K}$.

\begin{figure}[htbp]
\begin{center}
\hbox{
	\subfigure[Unit circle: $\alpha = (1,0)$]
		{\includegraphics*[trim={2.2cm 1.6cm 1.6cm 1.2cm},scale=0.30,clip]{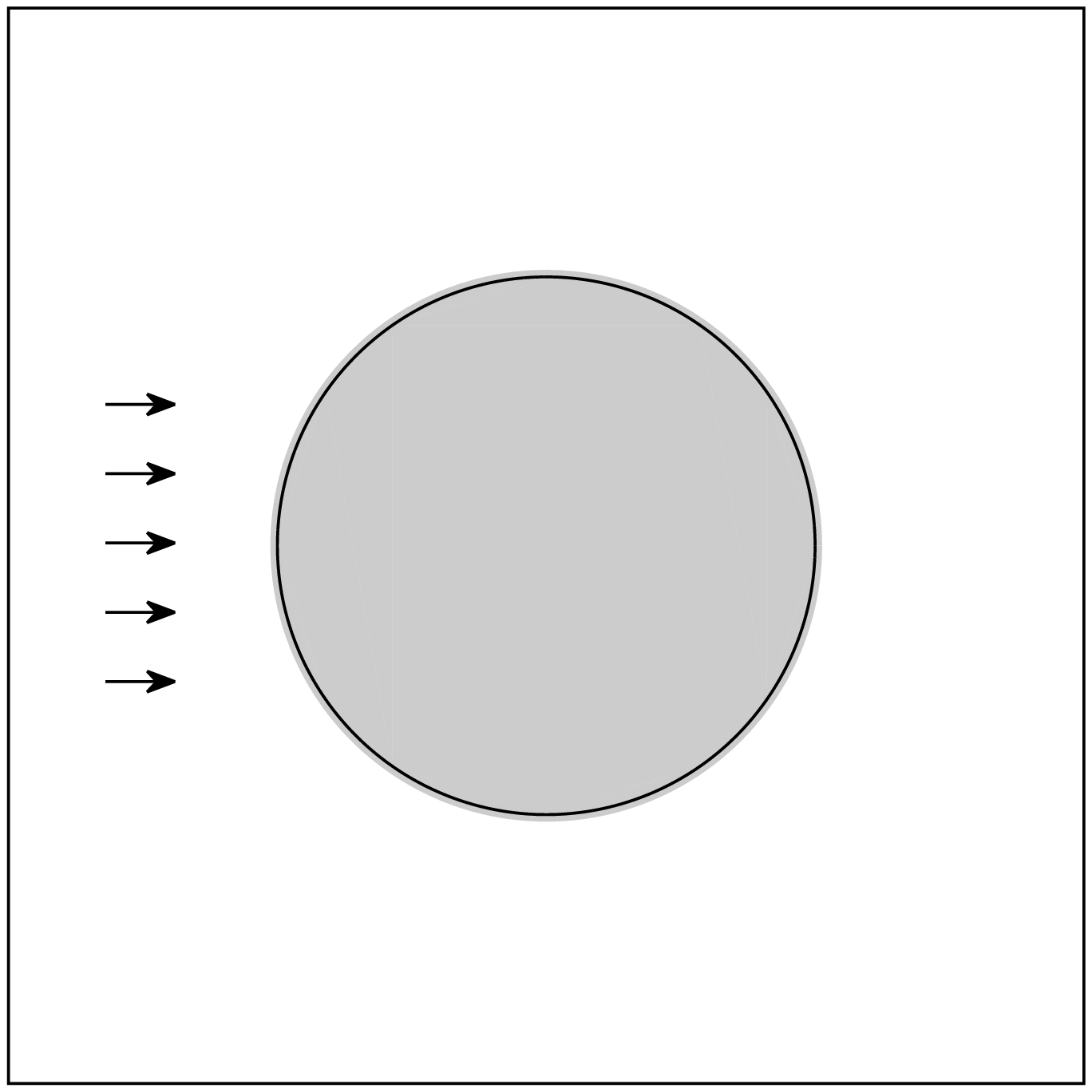}}
	\subfigure[Ellipse: $\alpha = (3,1)/\sqrt{10}$]
		{\includegraphics*[trim={2.6cm 1.3cm 1.8cm 1.2cm},scale=0.29,clip]{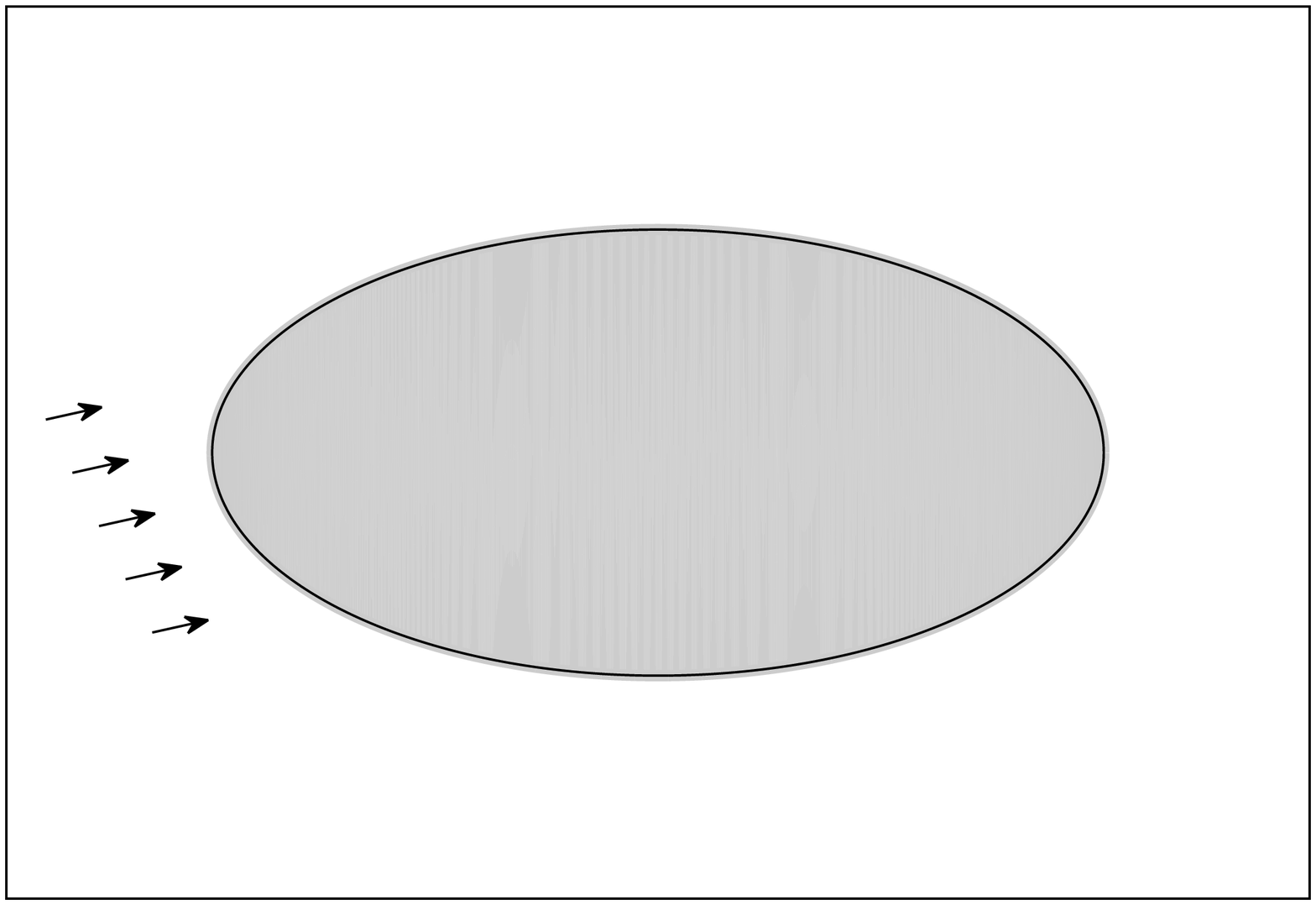}}
	\subfigure[Kite: $\alpha = (4,1)/\sqrt{17}$]
		{\includegraphics*[trim={2.2cm 1.4cm 1.6cm 1.2cm},scale=0.293,clip]{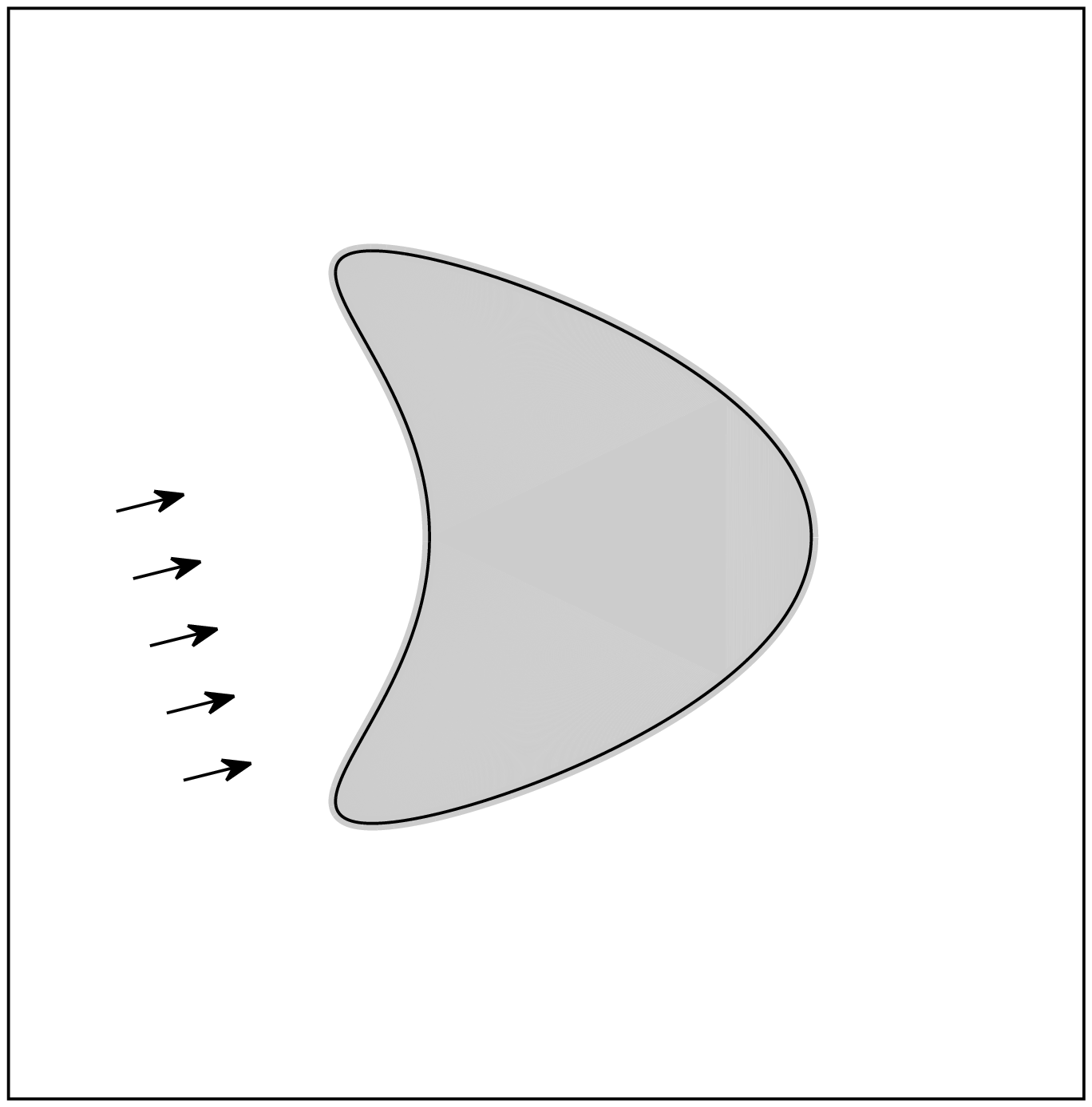}}
}
\end{center}
\caption{Single-scattering geometries and associated incidence directions $\alpha$.}
\label{fig:config}
\end{figure}

We consider three different single-scattering geometries (see Fig.~\ref{fig:config}) consisting of
the unit circle, the ellipse with major/minor axes (aligned with the $x/y$-axes) of $2/1$, and the kite
shaped obstacle given parametrically as $\{(\cos t + 0.65 \cos 2t - 0.65, 1.5 \sin t) \, : \, t \in [0,2\pi]\}$.
The unit circle is the standard example in the aforementioned references since
circles are the only two-dimensional obstacles for which explicit solutions are available
(through a straightforward Fourier analysis), and thus they allow an unquestionable performance
test for single-scattering solvers. As for the ellipse and kite shaped obstacles, we compare the
outcome of our numerical implementations with highly accurate reference solutions obtained by
a combination of the Nystr\"{o}m and trapezoidal rule discretizations applied to the CFIE
\cite{ColtonKress92}. Indeed, the double integrals appearing as the entries of Galerkin
matrices are also evaluated utilizing these rules for the inner integrals and the trapezoidal rule
for the outer integrals. In order to preserve the high-order approximation properties of these
numerical integration rules for smooth and periodic integrands, as in \cite{EcevitOzen16}, we additionally
utilize a smooth partition of unity confined to the regions on the boundary of the scatterer described
in Section~\ref{sec:3}. In each case, based on our experience in \cite{EcevitOzen16}, the number
of discretization points is chosen approximately as $10$ to $12$ points per wave length.

The results of our numerical experiments are presented in the following figures. They are arranged
so that the left panes display the support of direct summands forming the associated Galerkin
approximation spaces $\mathcal{A}_{\mathbf{d}}$ or $\mathcal{T}_{\mathbf{d}}$ we proposed in
\cite{EcevitOzen16} or their change of variables counterparts $\mathcal{A}_{\mathbf{d}}^{\mathcal{C}}$
or $\mathcal{T}_{\mathbf{d}}^{\mathcal{C}}$ we developed herein. On the other hand, the middle and
right panes depict, respectively, the corrseponding logarithmic relative errors
\[
	\log_{10} \left( \dfrac{\Vert \eta - \hat{\eta} \Vert_{L^2(\partial K)}}{\Vert \eta \Vert_{L^2(\partial K)}} \right)
	\quad
	\text{and}
	\quad
	\log_{10} \left( \dfrac{\Vert \eta - \hat{\eta} \Vert_{L^2(\partial K^{SR})}}{\Vert \eta \Vert_{L^2(\partial K^{SR})}} \right)
\]
versus the polynomial degrees used in each subregion (here
$\partial K^{SR}= \{ x \in \partial K : \alpha \cdot \nu(x) >0 \}$ is the shadow region).

Figures~\ref{fig:circle_mono_vs_chanvar} and \ref{fig:elips_tri_vs_chantri} concern
convex obstacles (the unit circle and the ellipse), and Figure~\ref{fig:kite} relates to
the non-convex single scattering configuration consisting of the kite. To prevent
repetitions, we have chosen to present a comparison of the solutions based on a utilization
of the approximation spaces
(i) $\mathcal{A}_{\mathbf{d}}$ and $\mathcal{A}_{\mathbf{d}}^{\mathcal{C}}$ for the
unit circle in Figure~\ref{fig:circle_mono_vs_chanvar}, (ii)
$\mathcal{T}_{\mathbf{d}}$ and $\mathcal{T}_{\mathbf{d}}^{\mathcal{C}}$ for the
ellipse in Figure~\ref{fig:elips_tri_vs_chantri}, and (iii)
$\mathcal{A}_{\mathbf{d}}^{\mathcal{C}}$ and $\mathcal{T}_{\mathbf{d}}^{\mathcal{C}}$
for the kite in Figure~\ref{fig:kite}. 

The first row in Figure~\ref{fig:circle_mono_vs_chanvar} displays the results based on
an implementation of $\mathcal{A}_{\mathbf{d}}$ with $m=1$ (which corresponds
to $8$ direct summands), and the second row to $\mathcal{A}_{\mathbf{d}}^{\mathcal{C}}$
with $J=6$ (giving rise to $6$ direct summands). As is apparent, solutions based on
$\mathcal{A}_{\mathbf{d}}^{\mathcal{C}}$ give rise to similar global accuracy as those obtained
by $\mathcal{A}_{\mathbf{d}}$ but with a reduction of $\%25$ in the total number of degrees
of freedom. Moreover, $\mathcal{A}_{\mathbf{d}}^{\mathcal{C}}$ provides significantly
improved accuracies in the shadow region.

\begin{figure}[htbp]
	\begin{center}
		\includegraphics[height=4.05cm, trim={2.2cm 0.3cm 0 0},clip]{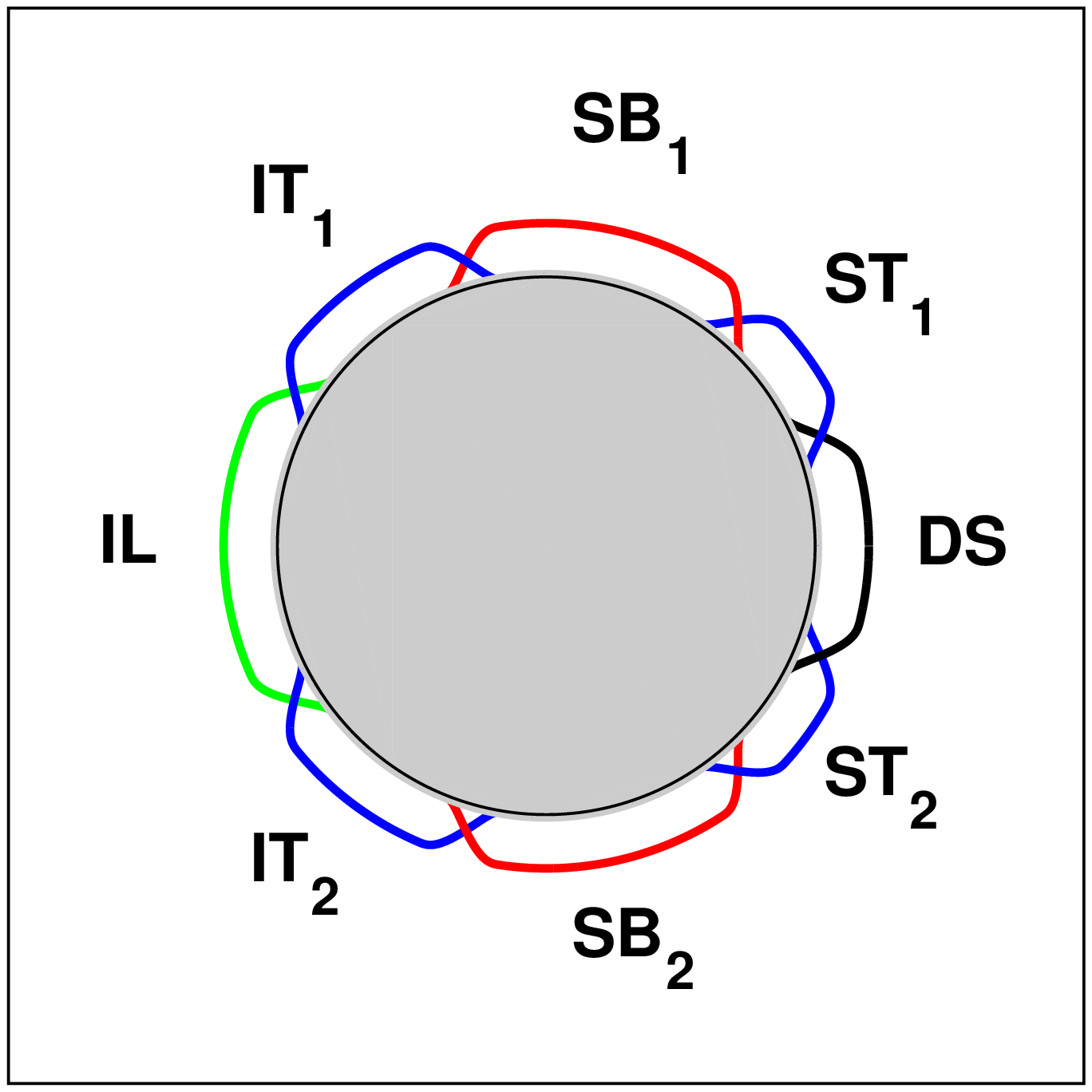} 
		\includegraphics[height=4cm, trim={3.5cm 0 15.5cm 0},scale=0.4,clip]{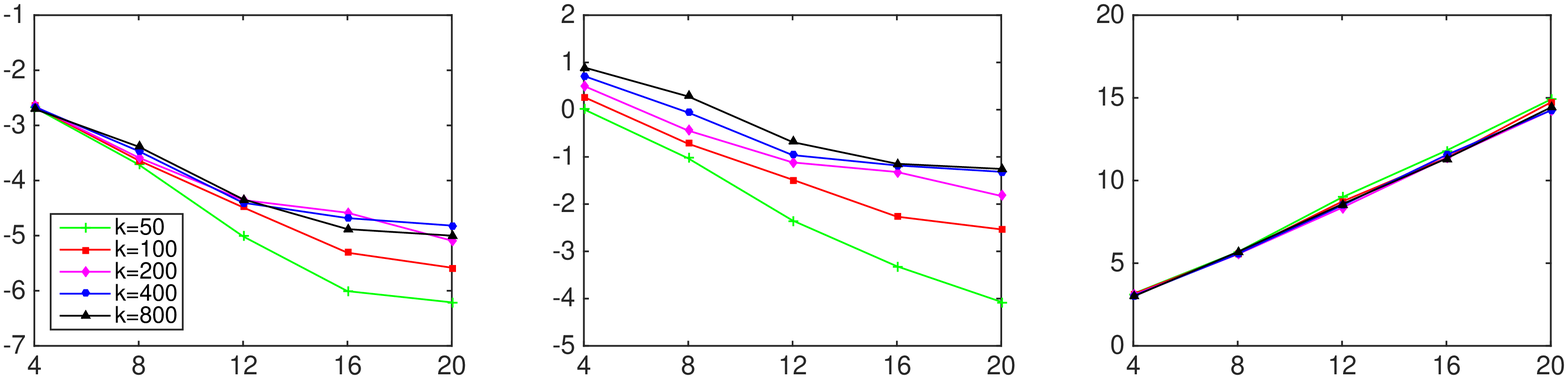} 
		\\
		\includegraphics[height=4.05cm, trim={2.2cm 0.3cm 0 0.},scale=0.3,clip]{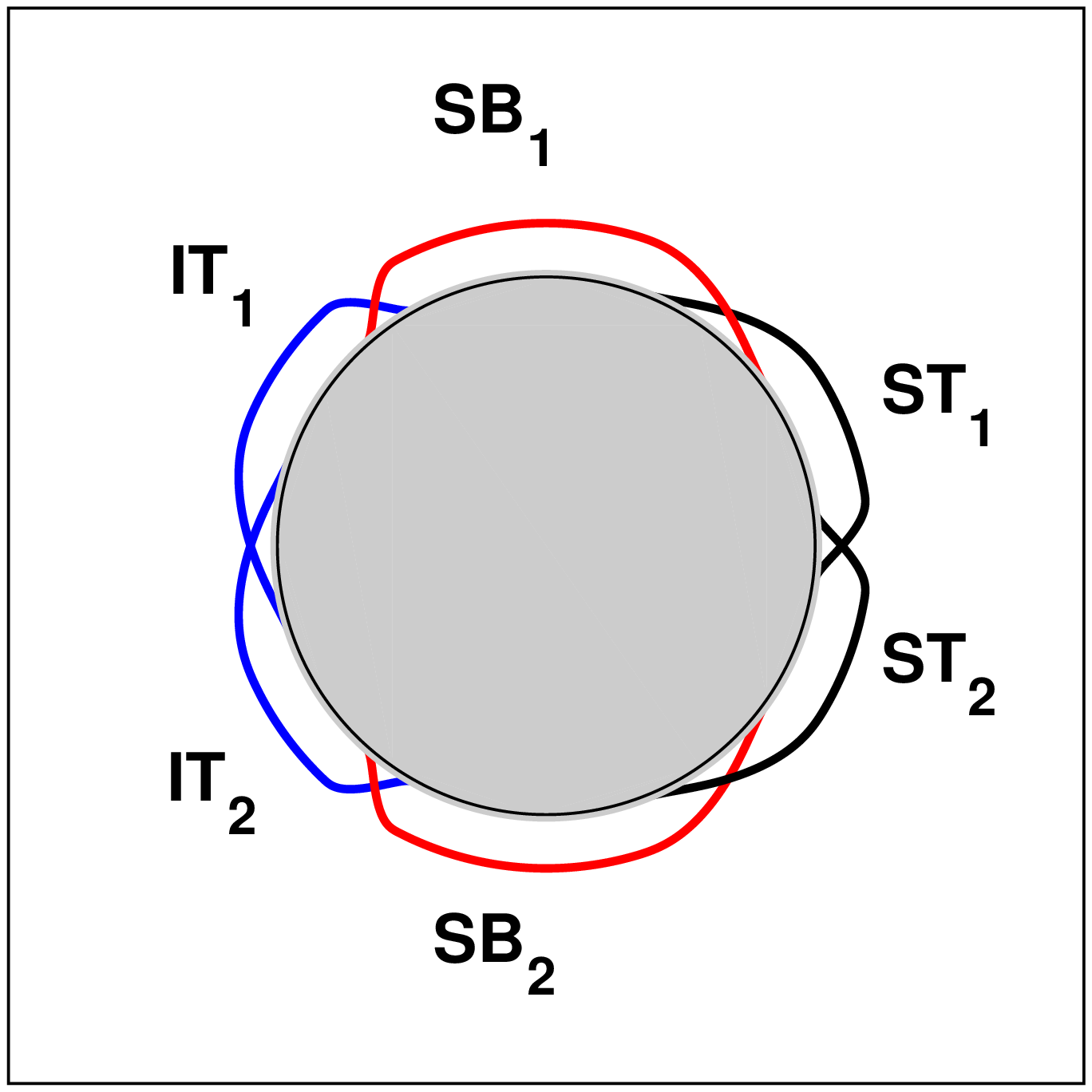}
		\includegraphics[height=4cm, trim={3.5cm 0 15.5cm 0},scale=0.4,clip]{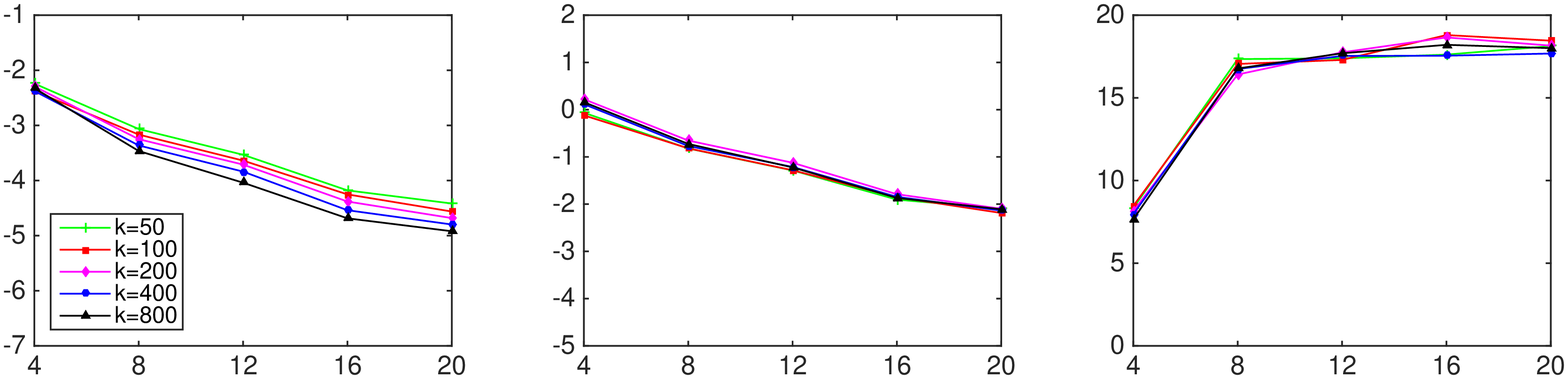}
	\end{center}
	\caption{$\mathcal{A}_{\mathbf{d}}$ (first row) vs. $\mathcal{A}_{\mathbf{d}}^{\mathcal{C}}$ (second row) for the unit circle.}
	\label{fig:circle_mono_vs_chanvar}
\end{figure}
\begin{figure}[htbp]
	\begin{center}
		\includegraphics[height=3.55cm, trim={2.2cm 0.3cm 0 0},clip]{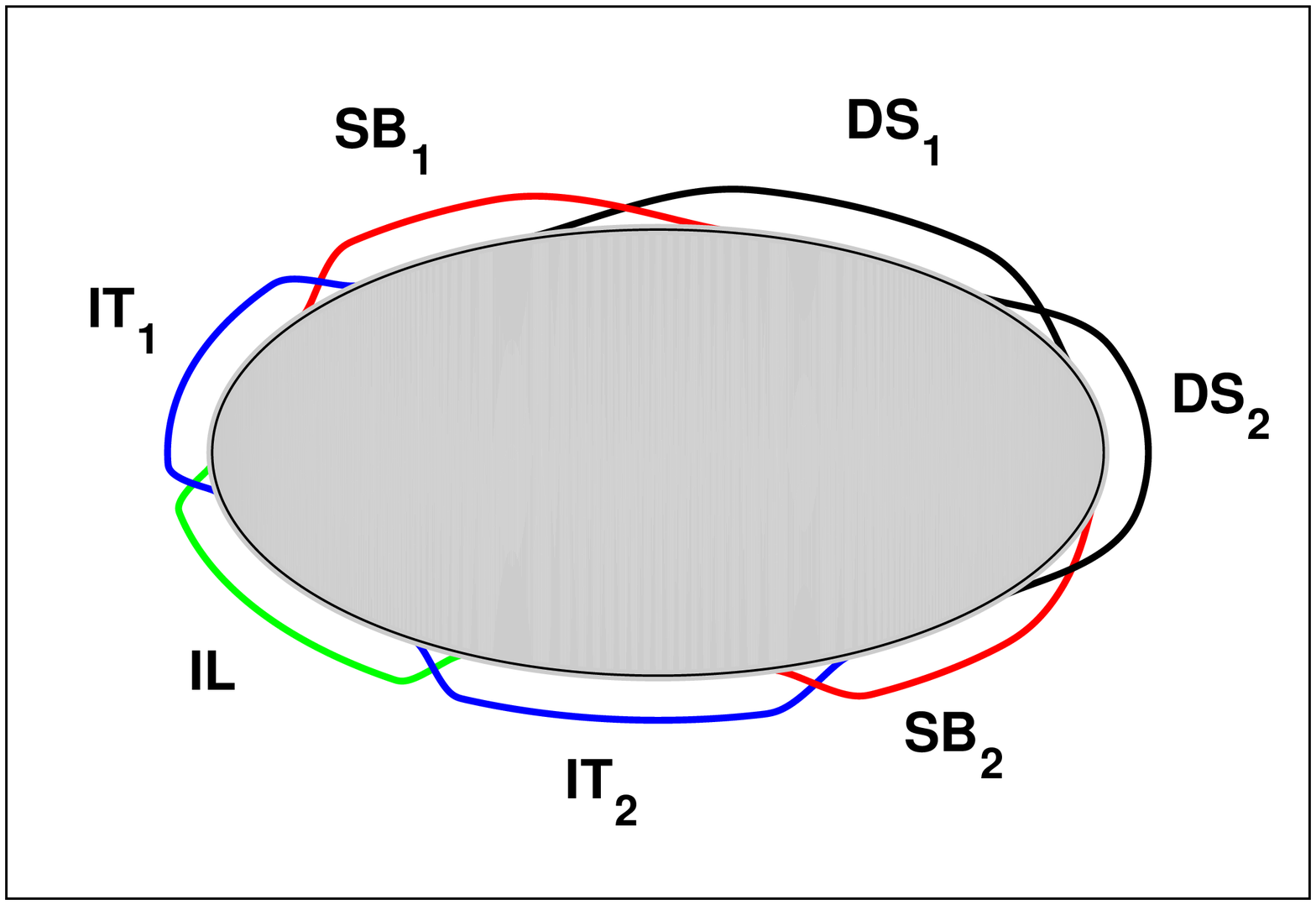}
		\includegraphics[height=3.50cm, trim={3.5cm 0 15.5cm 0},scale=0.4,clip]{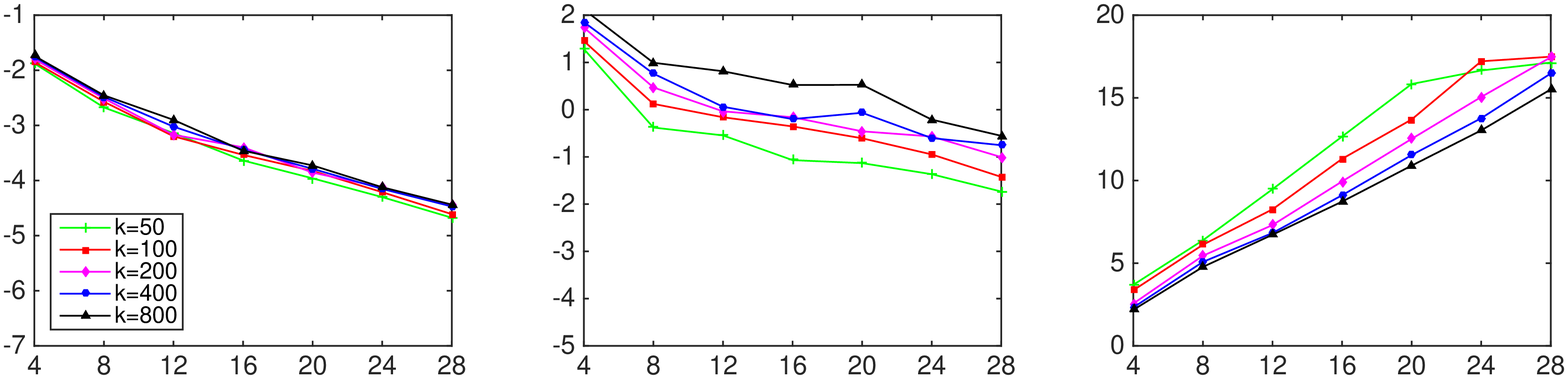}
		\includegraphics[height=3.55cm, trim={2.2cm 0.3cm 0 0},clip]{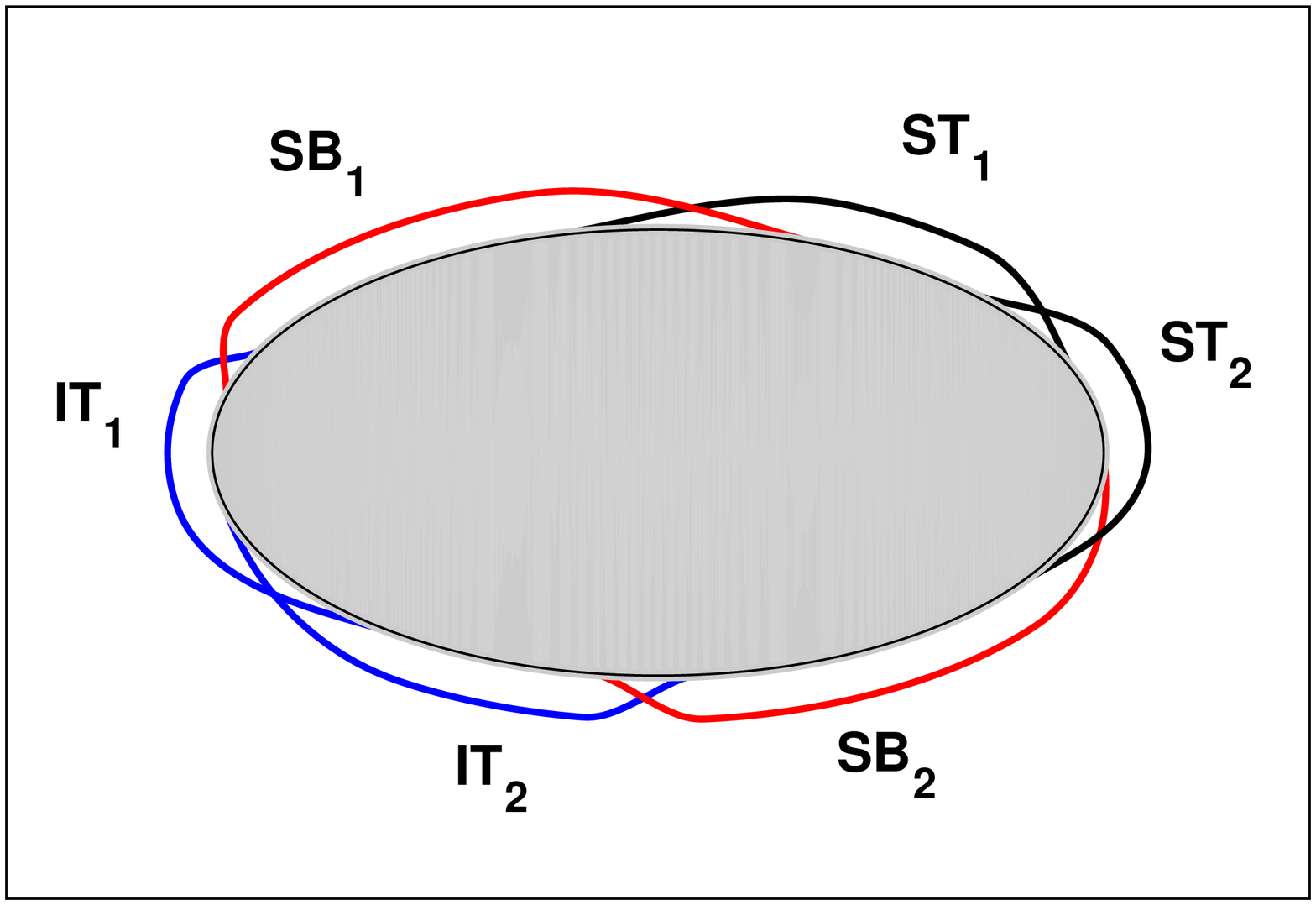}
		\includegraphics[height=3.50cm, trim={3.5cm 0 15.5cm 0},scale=0.4,clip]{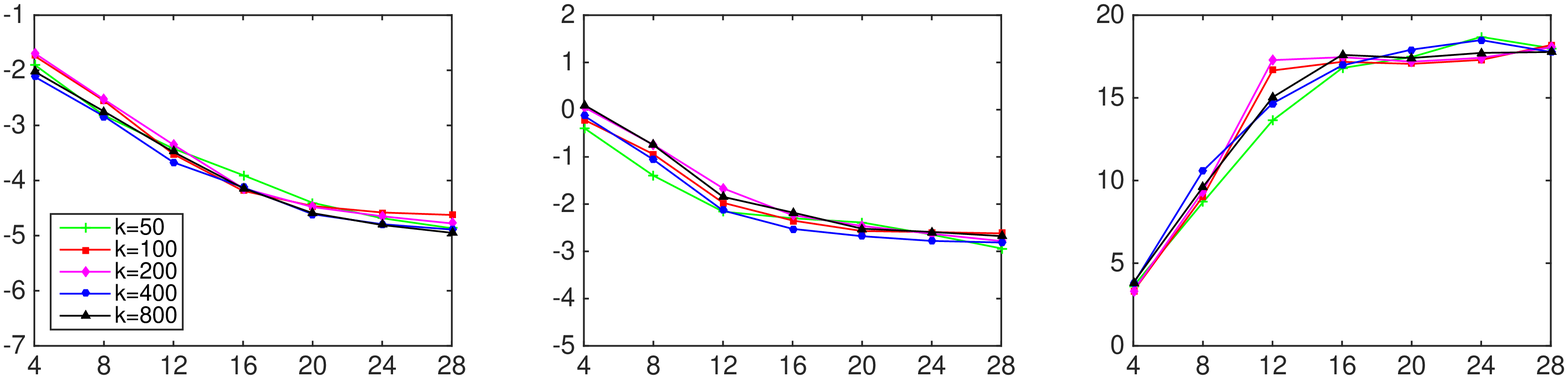}
	\end{center}
	\caption{A variant of $\mathcal{T}_{\mathbf{d}}$ (first row) vs. $\mathcal{T}_{\mathbf{d}}^{\mathcal{C}}$ (second row) for the ellipse.}
	\label{fig:elips_tri_vs_chantri}
\end{figure}

Figure~\ref{fig:elips_tri_vs_chantri} presents the numerical results associated with the ellipse obtained
by a variant of the $\mathcal{T}_{\mathbf{d}}$ spaces with $7$ direct summands (see \cite{EcevitOzen16}
for details) in the first row, and by the $\mathcal{T}_{\mathbf{d}}^{\mathcal{C}}$ spaces with $6$ direct
summands ($J=6$) in the second row. In this case, while the $\mathcal{T}_{\mathbf{d}}^{\mathcal{C}}$
spaces provide a slight improvement in the global accuracy over $\mathcal{T}_{\mathbf{d}}$ with savings
of about $\%14$ in the total number of degrees of freedom, the approximations they provide in the shadow
region are several orders of magnitude better.
\begin{figure}[htbp]
	\begin{center}
		\includegraphics[height=4.05cm, trim={2.2cm 0.3cm 0 0},clip]{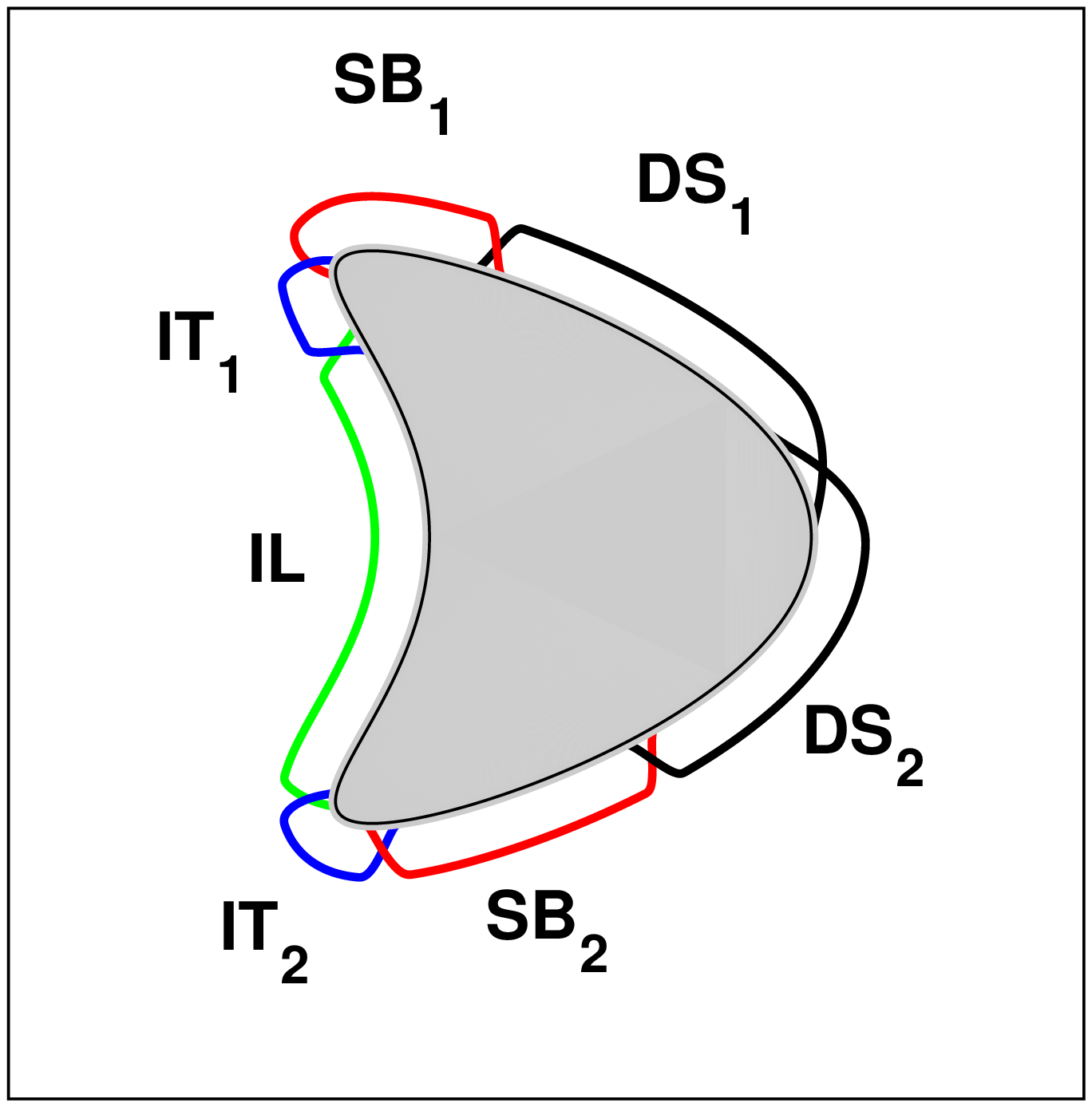}
		\includegraphics[height=4cm, trim={3.5cm 0 15.5cm 0},scale=0.4,clip]
		{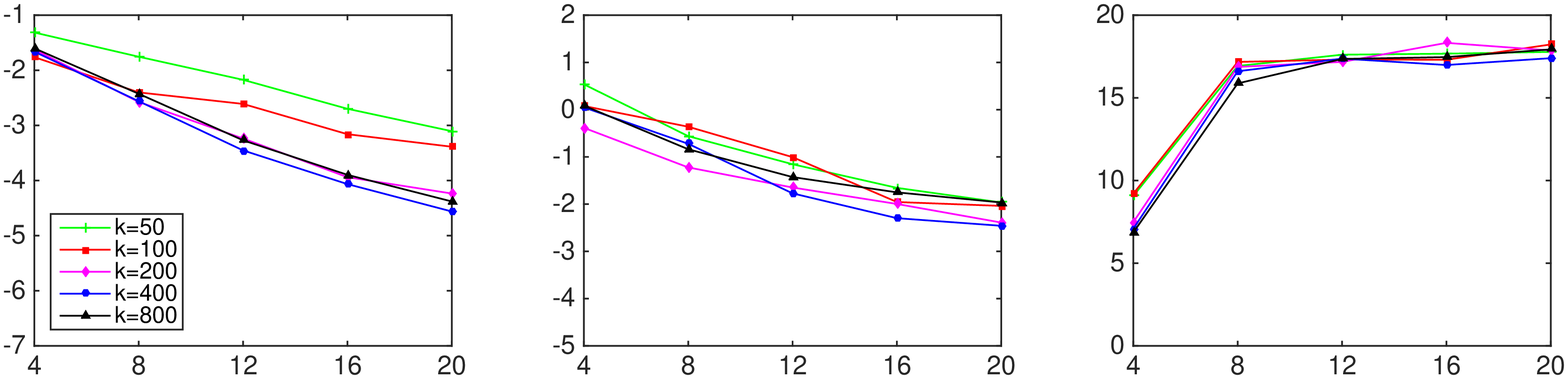}
		\includegraphics[height=4.05cm, trim={2.2cm 0.3cm 0 0},clip]{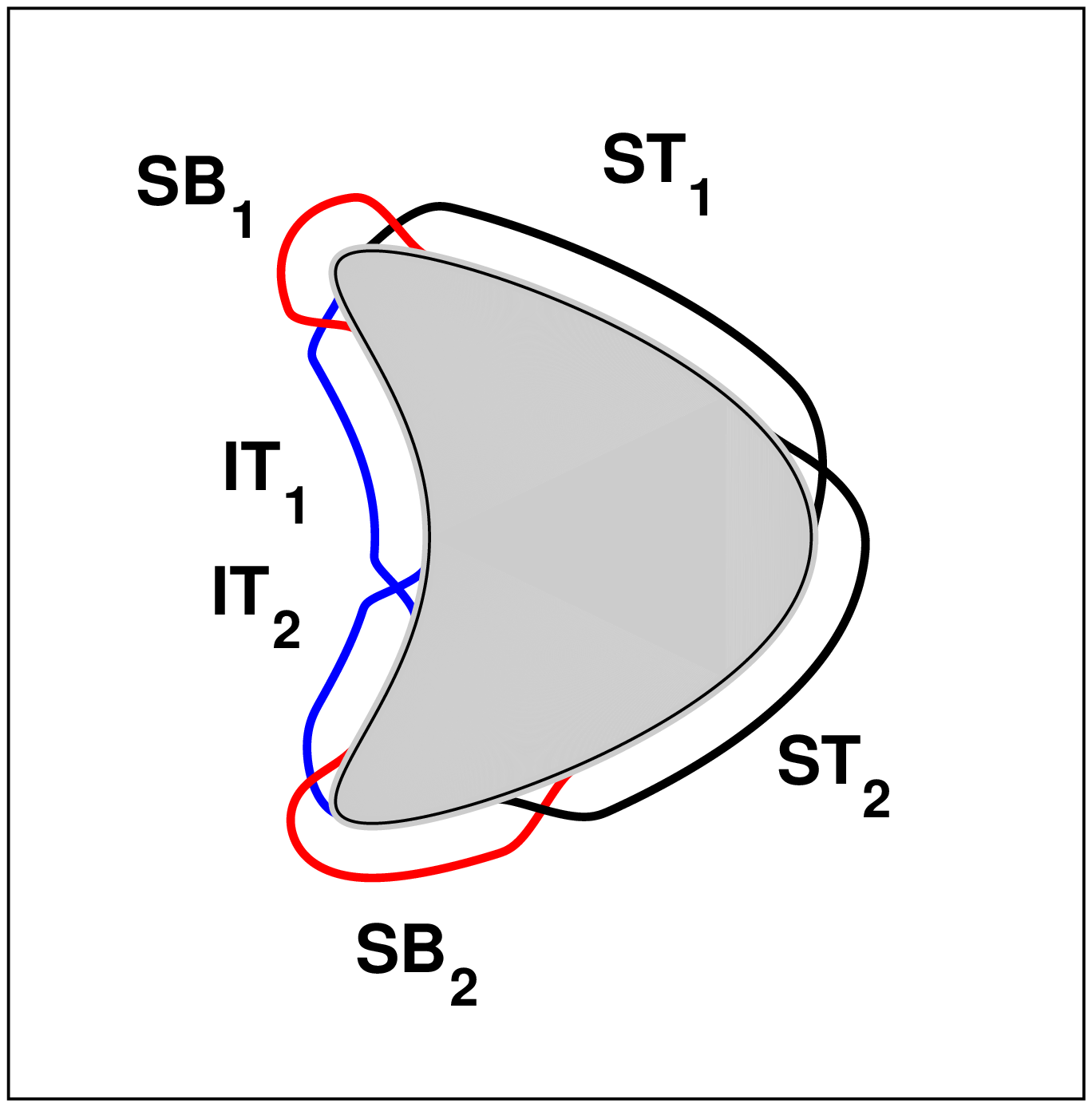}
		\includegraphics[height=4cm, trim={3.5cm 0 15.5cm 0},scale=0.4,clip]{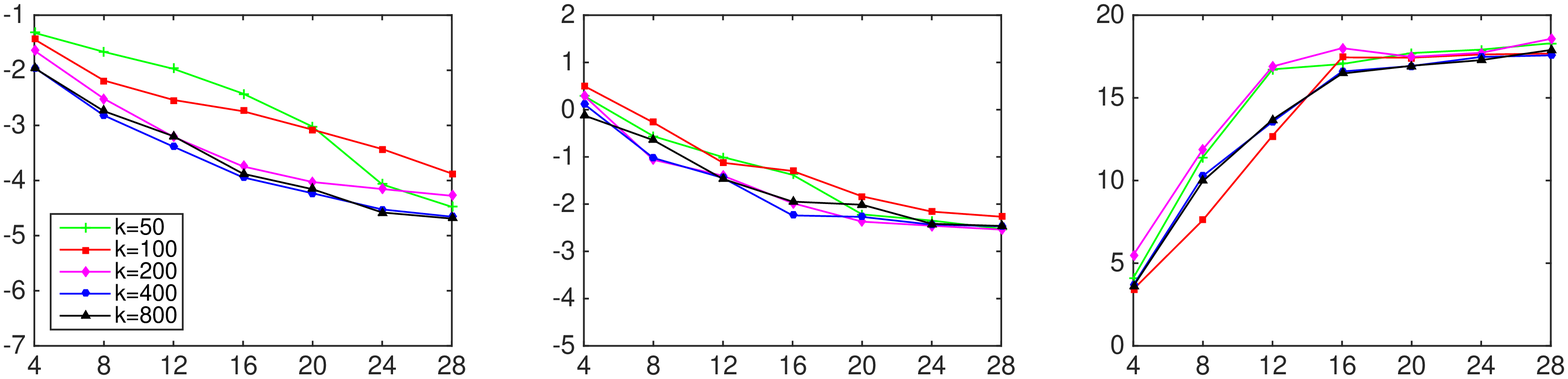}
	\end{center}
	\caption{$\mathcal{A}_{\mathbf{d}}^{\mathcal{C}}$ (first row) vs. $\mathcal{T}_{\mathbf{d}}^{\mathcal{C}}$ (second row) for the kite.}
	\label{fig:kite}
\end{figure}

Finally, considering the more general single-scattering geometry of the kite, in Figure~\ref{fig:kite}
we present a comparison of the solutions obtained by an implementation of the $\mathcal{A}_{\mathbf{d}}^{\mathcal{C}}$
and $\mathcal{T}_{\mathbf{d}}^{\mathcal{C}}$ spaces. In this case, motivated with our experience in
\cite{EcevitOzen16} we have constructed $\mathcal{A}_{\mathbf{d}}^{\mathcal{C}}$ spaces based on $7$
direct summands whereas we have used $J=6$ direct summands in forming $\mathcal{T}_{\mathbf{d}}^{\mathcal{C}}$
spaces (see the left-most pane in Fig.~\ref{fig:kite}). As Figure~\ref{fig:kite} displays, while the algebraic
$\mathcal{A}_{\mathbf{d}}^{\mathcal{C}}$ and trigonometric $\mathcal{T}_{\mathbf{d}}^{\mathcal{C}}$ Galerkin
approximation spaces based on changes of variables are both well adapted to more general non-convex single-scattering
geometries, the latter displays a slightly better performance for higher frequencies in terms of both
the global and shadow region errors.

\section{Conclusions}
In this paper, we proposed a class of \emph{Galerkin boundary element methods based on novel
changes of variables} for the solution of two-dimensional single-scattering problems. The Galerkin
approximation spaces, generated in the form of a direct sum of algebraic or trigonometric polynomial
spaces weighted by the oscillations in the incident field of radiation, are additionally coupled with
novel frequency dependent changes of variables in the transition regions. As we have shown,
this construction ensures that the global approximation spaces are perfectly matched with the
changes in the boundary layers of the solutions with increasing wave number $k$, and they provide
remarkable savings over their counterparts in \cite{EcevitOzen16} in regards to the total number
of degrees of freedom necessary to obtain a prescribed accuracy. Most notably, the schemes
proposed herein display the capability of delivering several orders of magnitude more accurate
solutions in the shadow regions.

\bibliographystyle{abbrv}
\bibliography{EcevitReferences}

\end{document}